\documentclass[12pt,letter,twoside,reqno]{amsart} 

\usepackage{tikz}
\usepackage{amsmath, verbatim}
\usepackage{amssymb,amsfonts,mathrsfs}
\usepackage[colorlinks=true, linkcolor=blue,citecolor=blue, bookmarks=true, pdfstartview=FitH, pagebackref=true]{hyperref}
\usepackage[driver=pdftex,margin=3cm,heightrounded=true,centering]{geometry}
\usepackage{enumitem}

\usepackage{mllocal} 

\newcommand{\scal}{\textup{scal}}
\newcommand{\vol}{\textup{vol}}

\newcommand{\ginit}{g_{\, \textup{init}}}
\newcommand{\U}{\mathscr{U}}

\newcommand{\vep}{\varepsilon}
\newcommand{\ubar}{\bar{u}}

\newcommand{\vl}{\textup{vol}}
\newcommand{\xt}{\widetilde{x}}
\newcommand{\dvinit}{dv_{\ginit}}

\newcommand{\bR}{\mathbb{R}}
\newcommand{\bZ}{\mathbb{Z}}
\newcommand{\inn}[2]{ \left\langle #1, #2 \right\rangle}

\newcommand{\cS}{\mathfrak{s}} 

\DeclareMathOperator{\hhok}{\mathcal{C}^{2k+\A}_{\textup{ie}}}
\DeclareMathOperator{\hhokk}{\mathcal{C}^{2(k+1)+\A}_{\textup{ie}}}

\DeclareMathOperator{\hospace}{\mathcal{C}^{\A}_{\textup{ie}}}
\DeclareMathOperator{\hhospace}{\mathcal{C}^{2+\A}_{\textup{ie}}}
\DeclareMathOperator{\ho}{\mathcal{C}^{\A}_{\textup{ie}}}
\DeclareMathOperator{\hoone}{\mathcal{C}^{1+\A}_{\textup{ie}}}
\DeclareMathOperator{\hoonek}{\mathcal{C}^{2k + 1+\A}_{\textup{ie}}}
\DeclareMathOperator{\hho}{\mathcal{C}^{2+\A}_{\textup{ie}}}

\DeclareMathOperator{\hhh}{\mathcal{C}^{4+\A}_{\textup{ie}}}

\DeclareMathOperator{\hhhs}{\mathcal{C}^{4+\sigma}_{\textup{ie}}}
\DeclareMathOperator{\hhod}{\mathcal{C}^{2+\A}_{\textup{ie}, 0}}

\numberwithin{equation}{section}

\begin{document}


\title[Long-time existence of the edge Yamabe flow]
{Long-time existence of the edge Yamabe flow}

\author{Eric Bahuaud}
\address{Department of Mathematics,
Seattle University,
Seattle, WA 98122,
USA}

\email{bahuaude (at) seattleu (dot) edu}

\author{Boris Vertman}
\address{Mathematisches Institut,
Universit\"at M\"unster,
48149 M\"unster,
Germany}
\email{vertman (at) uni-muenster (dot) de}

\subjclass[2000]{53C44; 58J35; 35K08}
\date{\today}

\begin{abstract}
{This article presents an analysis of the normalized Yamabe flow starting at 
and preserving a class of compact Riemannian manifolds with incomplete edge singularities
and negative Yamabe invariant. Our main results include uniqueness, long-time existence and convergence of the edge Yamabe flow starting at a metric with everywhere negative scalar curvature. Our methods include novel maximum principle results on the singular edge space without using barrier functions. Moreover, our uniform bounds on solutions
are established by a new ansatz without in any way using or redeveloping Krylov-Safonov estimates in the singular setting. 
As an application we obtain a solution to the Yamabe problem for incomplete edge metrics with negative Yamabe invariant using flow techniques. Our methods lay groundwork for studying other flows like the mean curvature flow as well as the porous medium equation in the singular setting. }
\end{abstract}

\maketitle
\tableofcontents

\section{Introduction and statement of the main result}\label{intro}

The normalized Yamabe flow (NYF) is a geometric evolution equation that evolves a Riemannian metric toward a metric of constant scalar curvature.  In the classical setting, let $(M,g)$ be a compact Riemannian manifold of dimension $m \geq 3$, let $\scal(g)$ and $\vol(g)$ denote the scalar curvature and volume respectively, and let $\rho(g) = \textup{vol}(g(t))^{-1} \int_M \scal(g(t)) \textup{dvol}_{g(t)}$ denote the average scalar curvature.  Then the volume normalized Yamabe flow starting at a metric $\ginit$ is 
\begin{equation}\label{NYF}  
\partial_{t} g(t) = \left(\frac{}{}\rho(t) - \scal(g(t))\right) \cdot g(t), \ g(0) = \ginit.\end{equation}

On a compact manifold, the behavior of this flow is rather well understood.  For any choice of initial metric, the flow exists for all time -- this fact was proved by Hamilton, who also introduced the Yamabe flow.  Ye proved convergence of the flow in the scalar negative, scalar flat and locally conformally flat scalar positive cases \cite{Ye}.  Building on work of Schwetlick and Struwe \cite{SS}, Brendle \cite{BrendleYF} was able to prove convergence in all cases for dimension $3 \leq m \leq 5$, and with technical assumptions on the Weyl curvature in dimensions $m \geq 6$.  See the survey \cite{Brendle} for further details. \medskip

The Yamabe flow equation preserves the conformal class of the initial metric $\ginit$ and thus may be written as a single nonlinear PDE for the conformal factor.  Further, the flow preserves the volume and decreases the average scalar curvature functional $\rho(t)$.  On a compact manifold, short-time existence is a consequence of a contraction mapping argument that uses classical parabolic Schauder estimates (see for example \cite{LSU:LAQ}). As explained in \cite{Ye}, long-time existence and convergence of the Yamabe flow is easiest in the case $\scal(\ginit)<0$.  Briefly, if there is a maximal finite time of existence, then the maximum principle may be used (this is where the sign condition $\scal(\ginit)<0$ arises) to obtain a uniform $L^{\infty}$ estimate for the conformal factor up to the maximal time.  A regularity result by Krylov-Safonov \cite{KrylovSafonov} then allows us to improve this to a H\"older bound, and parabolic regularity implies that the conformal factor then extends to the maximal time where the flow may be restarted.  Thus the flow exists for all time, and convergence rapidly follows from an analysis of the average and total scalar curvature functionals.
\medskip

In this paper we continue the investigation of the Yamabe flow that we commenced in \cite{BV} where we established short-time existence for the Yamabe flow within a specified class of compact Riemannian spaces with incomplete edge metrics.  Preservation of the edge structure along the flow can roughly be thought of as a boundary condition.  We present a subsequent study of uniqueness and long-time existence for this edge Yamabe flow in the case of initial edge metric with negative scalar curvature.  Needless to say in a singular setting none of the analysis described in the compact setting can be taken for granted.  Classical parabolic Schauder estimates do not directly carry over to the singular setting and even short-time existence of the Yamabe flow was non-trivial. Moreover, the maximum principle needs to be suitably modified to incorporate the singular edge boundary, and we had to develop a suitable replacement for the Krylov-Safonov estimate.\medskip  

The incomplete edge metrics of interest are certain perturbations of rigid edge metrics.  Given a compact manifold with boundary $(\widetilde{M},\partial M)$, let $x$ be a simple defining function for the boundary, and assume the boundary is a fibration with base $(B,g^B)$ and fibre $(F,g^F)$ so that $\phi: (\partial M, g^F + \phi^* g^B) \to (B,g^B)$ is a Riemannian submersion. 
We assume $\dim F \geq 1$ throughout the paper. We write $M:=\widetilde{M} \backslash \partial M$ for the open interior of $\widetilde{M}$.
The rigid edge metrics on $M$ under consideration here are then expressible in a neighborhood 
of $\partial M$ in the form
\[ g_0 = dx^2 + x^2 g^F + \phi^* g^B. \]
The metrics we work with, which we call \textit{feasible}, are then higher order perturbations of rigid metrics satisfying certain additional conditions. These higher order perturbations are smooth on $\widetilde{M}$ and vanishing at 
$\partial M$. See the next section for more details.  \medskip

The key point of the feasibility requirements we impose above is that we may obtain the asymptotics of the heat kernel of the Friedrichs extension of the Laplace Beltrami operator $\Delta^{g}$, and thus study its mapping properties between functions spaces.  In fact, parallel with the recent developments for other geometric problems on singular spaces, we have introduced hybrid incomplete edge H\"older spaces that only keep track of certain derivatives.  Let $\hospace$ be the space of functions with finite (parabolic) H\"older norm, where distance on $M$ is measured with respect to $g$ (see Definition \ref{holder-def} below).  The basic space on which we prove existence of the Yamabe flow is $\hho$ where
\[ \hho(M \times [0,T]) = \{ u \in \hospace(M \times [0,T]) \; | \;\Delta u, x^{-1} \mathcal{V}_e u, \partial_t u \in \hospace \}, \]
with an appropriate norm, where $\mathcal{V}_e$ denotes first order edge vector fields.  In other words, our second-order hybrid spaces only keep track of $\Delta u$ instead of the full Hessian $\nabla^2 u$.  This is related to the failure of elliptic regularity in 
H\"older spaces which control the full Hessian in this singular setting. See \cite{MRS} for an excellent discussion of these issues.
\medskip

\subsection{Statement of the main results}
Let us now state our main results.   First, we point out a new version of a maximum principle 
on singular spaces which specifies regularity of a function instead of reducing the statement
to the classical setting by a smart choice of barrier functions. We prove the following.

\begin{theorem} Let $g$ be a feasible incomplete edge metric with Laplacian $\Delta$.  Suppose $u \in \hho$ attains its minimum (respectively maximum) at 
some point $p\in \widetilde{M}$. In particular, we do not require $p$ to lie in the interior of the edge space. 
Then $(\Delta u)(p) \geq 0$ $($respectively $(\Delta u)(p) \leq 0 )$.
\end{theorem}

A direct consequence of this result are the various versions of a parabolic maximum principle,
which we use decisively at the various steps in the argument. Before stating our subsequent results, we recall an important ingredient from the resolution of the Yamabe problem.  Recall on a compact manifold of dimension at least $3$
the (normalized) total scalar curvature function is defined as 
\[ \cS(g) := \frac{1}{\vol(g)^{\frac{m-2}{m}}} \int_M \scal(g) d\vol. \]
The Yamabe invariant is then defined as
\[ \mathcal{Y}( [g] ) = \mathcal{Y}( u^{\frac{4}{m-2}} g) = \inf \left\{  \cS( u^{\frac{4}{m-2}}g ): u \in H^{1,2}(M), u > 0 \right\}. \]
We thus introduce an invariant adapted to our edge setting
\[ \nu([g]) = \inf\left\{  \cS( u^{\frac{4}{m-2}} g ): u \in \hho(M) , u > 0 \right\}. \]
Like the case of compact manifolds, the sign of $\nu$ plays an important role in the analysis.  We say that a conformal class $[g]$ is scalar positive, negative or zero if $\nu([g])$ is positive, negative or zero, respectively.  We prove the following singular analogue of a classical result.
\begin{thm} \label{negative}
Let $g$ be a feasible incomplete edge metric with conformal Laplacian 
$\Box^g u :=  - \frac{4(m-1)}{(m-2)} \, \Delta^g u+  \scal (g) u$,
and $\scal(g) \in \hhh(M)$.  The following are equivalent.
\begin{enumerate}
\item The first eigenvalue of $\Box^g$ is positive (respectively negative or zero).
\item There exists a metric $\widetilde{g} = u^{\frac{4}{m-2}} g$ such that $\scal(\widetilde{g}) > 0$\\
 (respectively $<0$ or $=0$).
\item $[g]$ is scalar positive (respectively scalar negative or scalar zero).
\end{enumerate}
\end{thm}

\medskip
\noindent
Related to the Yamabe flow in particular, we have
\begin{thm}\label{BV}
Let $(M,\ginit)$ be a feasible incomplete edge space of $\dim M \geq 3$.  Suppose $\scal(\ginit) \in \hhhs$ for some $\sigma \in (0,1)$.  Then there exists $\alpha \in (0,\sigma)$ such that the normalized Yamabe flow has a unique solution for a short-time in $\hho$.  Moreover, if the initial scalar curvature $\scal(\ginit) < 0$ then the normalized Yamabe flow starting at $\ginit$ has a unique solution on the infinite time interval $[0,\infty)$,
and converges to a metric of negative constant scalar curvature.
\end{thm}

As a corollary of the last two theorems, we conclude with a solution to the Yamabe problem in the singular edge setting for scalar negative metrics, using parabolic methods. Note that \cite{ACM} contains a proof of this fact using elliptic PDE techniques. We point out 
that our parabolic methods lay groundwork for studying other flows like the mean curvature flow as well as the porous medium equation in the singular setting. 

\begin{cor}\label{BV-cor} 
Suppose $\dim M \geq 3$ and $\ginit$ is a feasible edge metric, such that $\scal(\ginit) \in \hhhs$ and $\scal(\ginit) < 0$.  Then there exists an incomplete edge metric $g$ conformal to $\ginit$ with constant negative scalar curvature.
\end{cor}

Our analysis proceeds by rewriting the Yamabe flow as a non-linear parabolic equation for the conformal factor.  There are several common alternatives for doing so, for example we may write $g=u^{\frac{4}{m-2}}\ginit$ or $g=e^{2v}\ginit$ and rewrite the (unnormalized) Yamabe flow $\partial_t g = - \scal(g) g$ as a scalar equation for $u$ or $v$ 
\begin{equation}\label{YF-conf} 
\begin{split}
&\partial_t u^{\frac{m+2}{m-2}} = \frac{m+2}{m-2}  \left( 
(m-1)\Delta^{\ginit} u - \frac{m-2}{4} \scal(\ginit) u\right), \quad u(t=0)=1, \\
&\partial_t e^{2v} = - 2(m-1) \Delta^{\ginit} v + (m-2)(m-1) 
|\nabla v|^2 - \scal(\ginit), \quad v(t=0) = 0.
\end{split}
\end{equation}
While we used the latter equation in our short-time existence analysis in \cite{BV}, we employ the former equation here to more closely match the work of Ye that we adapt.  For the remainder of the paper, unless otherwise stated, we use the convention that $\Delta$ denotes the negative Laplace Beltrami operator of the initial metric $(M,\ginit)$.  We write $N = \frac{m+2}{m-2}$ and $c(m):= \frac{m+2}{4}$.   The normalized Yamabe flow then devolves to 
\begin{equation} \label{flow2}
 \partial_t u^N = (m-1) N \Delta u - c(m) \scal(\ginit) u + c(m) \rho \, u^N.
\end{equation}

Let us finally point out that feasible incomplete edge metrics $\ginit$ with everywhere strictly negative 
scalar curvature in fact exist. Take for example the recent work by Jeffres, Mazzeo and Rubinstein \cite{JMR},
where the authors establish existence of K\"ahler Einstein metrics of constant negative Ricci curvature
with incomplete edge singularities along divisors of codimension two. Such a metric is in particular of constant negative scalar curvature.  By continuity of  scalar curvature, any small not necessarily K\"ahler smooth perturbation of such a metric of sufficiently high order at the edge singularity yields a Riemannian incomplete edge metric with everywhere strictly negative and bounded
scalar curvature. Feasibility assumptions in the codimension two setting are easy to check.
\medskip

We now give a brief outline of the proof of Theorem \ref{BV}, and explain the structure of the paper.  We review the basics of edge geometry in \S \ref{edge-section}.  From our previous work \cite{BV}, corresponding to any initial choice of feasible edge metric, $\ginit$, with $\scal(\ginit) \in \hhhs$ for some $\sigma \in (0,1)$, there exists an $\alpha \in (0,\sigma)$ and a positive solution to NYF as in equation \eqref{YF-conf} exists for a short time in $\hho(M \times [0,T))$. \medskip

\begin{enumerate}[leftmargin= 0.5 cm]
\item[1)] We prove that solutions in $\hho$ are unique. This requires setting up an appropriate edge maximum principle which we undertake in \S \ref{sec:maxprin}.  A standard uniqueness argument for the conformal factor then applies and is detailed in \S \ref{unique-section}. \medskip

\item[2)] It now makes sense to speak of the maximum interval of existence $[0,T_M)$ of the flow.  If, by way of contradiction, $T_M < \infty$, then we obtain $L^{\infty}$ estimates for $u$ on $[0,T_M)$ using the same maximum principle technique as in \cite{Ye}.  It is here that we use the assumptions on the sign of the scalar curvature of $\ginit$.  The argument is given in \S \ref{scal-section} and the first part of \S \ref{uniform-section}. \medskip

\item[3)] We then conclude $L^{\infty}$ estimates for $\partial_t u$ by an interesting new observation.  Combined with parabolic regularity, we are able to obtain uniform $\hho$ estimates for $u$ on $[0,T_M)$.  This is a major point of departure from the classical case which appealed to Krylov-Safonov estimates.  We prove the necessary parabolic regularity for variable coefficient equations in \S \ref{parabolic-section}, and the estimates for higher derivatives of $u$ is given in the second part of \S \ref{uniform-section}.  The regularity arguments are delicate and require additional mapping properties of the heat kernel. \medskip

\item [4)] The uniform estimates of $u$ allow us to extend $u$ to $t = T_M$ and we may then restart the flow.  This contradiction establishes $T_M = \infty$.  This is argued in \S \ref{long-section}. \medskip
\end{enumerate}

Finally, we prove Theorem \ref{BV} by arguing convergence in \S \ref{convergence-section}.  The proof of Theorem \ref{negative} is given in \S \ref{negative-section}. We emphasize that all of the analysis above, including restarting the flow in step 4, is done with respect to the original feasible metric $\ginit$.  While the NYF preserves the edge structure (see Section \ref{invariance} below), at present we do not know if the evolving metric remains feasible even for a short-time. 
\medskip

We conclude by mentioning some related work.  The work of Akutagawa, Carron and Mazzeo \cite{ACM} studies the Yamabe problem on stratified spaces from a variational point of view, and the paper of Mazzeo, Rubinstein, Sesum \cite{MRS} studies the Ricci flow on surfaces with conic singularities.  Both of these works, like the present one, make extensive use of geometric microlocal analysis techniques.  See also \cite{AkutagawaBotvinnik} and \cite{JeffresRowlett}.  An alternative ansatz using maximal regularity has been employed by Shao \cite{Shao} to establish 
short-time existence of Yamabe flow on manifolds with isolated conical singularities. The principle of maximal regularity has also
been applied by Schrohe and Roidos \cite{Schrohe} in the context of the porous media equation on isolated conical singularities.
The work of Yin \cite{Yin} approaches the Ricci flow on surfaces with conic singularities by yet another alternative ansatz.  Finally we mention flows of incomplete metrics that regularize the edge, see the work of Giesen and Topping \cite{Topping, Topping2} and Simon \cite{MilesSimon} for example. \medskip

\textbf{Acknowledgements.} We are happy to thank Rafe Mazzeo and Burkhard Wilking for helpful conversations.  The authors are grateful the CIRM program ``Evolutions equations in singular spaces'', April 2016, where this work was completed.  Travel of the first author supported by conference grant NSF-DMS1600014.

\section{Review of the singular edge geometry}\label{edge-section}

\subsection{Incomplete edge spaces} 

Consider a compact stratified space  $\overline{M}$
which is comprised of a single open stratum $M$ of dimension $m\geq 3$ and 
a single lower dimensional stratum $B$, the edge singularity, which is a closed manifold
of dimension $b$. By definition of stratified spaces there exists an open 
neighborhood $U\subset \overline{M}$ of $B$
together with a radial function $x:U \to [0,1)$, such that $U\cap M$ 
is the total space of a smooth fibre bundle over $B$ with the trivial fibre 
given by a truncated cone $\mathscr{C}(F)=(0,1)\times F$ 
over a compact smooth manifold $F$ of dimension $n$.
We assume $n\geq 1$ throughout this paper. The radial function $x$ 
restricts to a radial function of that cone on each fibre.
\medskip

In order to define continuity up to the edge, we may resolve the singular stratum 
$B$ in $\overline{M}$ to define a compact manifold $\widetilde{M}$ with boundary $\partial M$. 
$\partial M$ is then the total space of a fibration $\phi: \partial M \to B$ with the fibre $F$. 
Assume $\dim F\geq 1$ throughout the paper. Under the resolution, the neighborhood $U$ lifts to a collar neighborhood 
$\U \subset \widetilde{M}$, which is a smooth fibration of cylinders 
$[0,1)\times F$ over $B$ with the radial function $x$.
Clearly $M=\widetilde{M} \backslash \partial M$.

\begin{defn}\label{d-edge}
A Riemannian manifold with an incomplete edge singularity is the open stratum $M$ together with 
a Riemannian metric $g$, such that over $\U\backslash \partial M$ the metric attains the form $g_0+h$ with
$$g_0 =dx^2+x^2 g^F+\phi^*g^B,$$
where $g^B$ is a Riemannian metric on the closed manifold $B$, 
$g^F$ is a symmetric 2-tensor on the fibration $\partial M$ restricting to an 
isospectral family of Riemannian metrics on the fibres $F$, and $|h|_{g_0}$ is smooth on $\U$ and 
vanishes at $x=0$. 
\end{defn}

In our previous work \cite{BV} we had to pose further assumptions on the Riemannian metric $g$
in order to establish parabolic Schauder estimates in the incomplete edge setting and 
derive short-time existence of the edge Yamabe flow. \medskip

We summarize this and the other assumptions from our previous work 
into the notion of feasible edge metrics.

\begin{defn}\label{def-feasible}
Let $(M,g)$ be a Riemannian manifold with an edge metric. This metric $g=g_0+h$ is said to be 
feasible if the following additional conditions are satisfied 
\begin{enumerate}
\item $|h|_{g_0}=O(x^2)$ as $x\to 0$, 
\item $\phi: (\partial M, g^F + \phi^*g^B) \to (B, g^B)$ is a Riemannian submersion, 
\item the lowest non-zero eigenvalue $\lambda_0 > 0$ of $\Delta_F$ satisfies $\lambda_0 > n$.
\end{enumerate}
\end{defn} 

We refer the reader to \cite[\S 1.1]{BV} for more details on these 
feasibility assumptions and elaboration where these conditions have been 
used in our argument. In a recent work by the second named author 
\cite{Ricci-Vertman} on the Ricci flow for edge manifolds, the last 
assumption of Definition \ref{def-feasible} has been dropped and in fact can 
be dropped in the present setting as well. 

\subsection{Microlocal heat kernel asymptotics} 

Let $\Delta$ denote the Friedrichs self-adjoint extension of 
the (negative) Laplace Beltrami operator on $(M,g)$, with domain $\dom (\Delta)$. 
The corresponding heat operator of $\Delta$ acts as an integral convolution operator 
on $u(t,\cdot)\in \dom (\Delta)$ 
\begin{equation} \label{eqn:hk-on-functions}
e^{t\Delta} u(t,p) = \int_0^t \int_M H\left( t-\wt, p,\widetilde{p} \right) 
u(\wt, \widetilde{p}) \dv (\widetilde{p}) \, d\wt,
\end{equation}
and solves the inhomogeneous heat problem
\begin{equation*}
(\partial_t - \Delta) \w(t,p)  = u(t,p), \ \w(0,p)=0,
\end{equation*}
for any $u(t,\cdot)\in \dom (\Delta)$. The heat kernel $H$ is a function on $M^2_h=\R^+\times \widetilde{M}^2$
with non-uniform behavior along certain submanifolds of $M^2_h$. 
Consider local coordinates $(y)$ on $B$, lifted to $M$ and then extended inwards to the interior 
of $M$. Let coordinates $(z)$ restrict to local coordinates on fibres $F$. Then, the local 
coordinates near the corner in $M^2_h$ are given by $(t, (x,y,z), (\widetilde{x}, \wy, \widetilde{z}))$, 
where $(x,y,z)$ and $(\widetilde{x}, \wy, \widetilde{z})$ are two copies of coordinates fixed above near the edge. 
The heat kernel $H(t, (x,y,z), (\wx,\wy,\wz))$ behaves non-uniformly at the submanifolds
\begin{align*}
&A =\{ (t, (x,y,z), (\wx,\wy,\wz))\in M^2_h \mid t=0, \, x=\wx=0, \, y= \wy\}, \\
&D =\{ (t, p, \widetilde{p})\in M^2_h \mid t=0, \, p=\widetilde{p}\}.
\end{align*}
The non-uniform behavior of $H$ is resolved by blowing up the submanifolds $A$ and $D$ 
appropriately, such that the heat kernel lifts to a polyhomogeneous function
in the sense of the following definition, cf. \cite{Mel:TAP}.

\begin{defn}\label{phg}
Let $\mathfrak{W}$ be a manifold with corners and denote by $\{(H_i,\rho_i)\}_{i=1}^N$ an enumeration 
of its (embedded) boundaries with the respective defining functions. For any multi-index $b= (b_1,
\ldots, b_N)\in \C^N$ we write $\rho^b = \rho_1^{b_1} \ldots \rho_N^{b_N}$.  Denote by 
$\mathcal{V}_b(\mathfrak{W})$ the space of smooth vector fields on $\mathfrak{W}$ which lie
tangent to all boundary faces. A distribution $\w$ on $\mathfrak{W}$ is said to be conormal,
if $\w$ is a restriction of a distribution across the boundary faces of $\mathfrak{W}$, 
$\w\in \rho^b L^\infty(\mathfrak{W})$ for some $b\in \C^N$ and $V_1 \ldots V_\ell \w \in \rho^b L^\infty(\mathfrak{W})$
for all $V_j \in \mathcal{V}_b(\mathfrak{W})$ and for every $\ell \geq 0$. An index set 
$E_i = \{(\gamma,p)\} \subset {\mathbb C} \times {\mathbb N_0}$ 
satisfies the following hypotheses:

\begin{enumerate}
\item $\textup{Re}(\gamma)$ accumulates only at $+\infty$,
\item for each $\gamma$ there exists $P_{\gamma}\in \N_0$, such 
that $(\gamma,p)\in E_i$ for all $p \leq P_\gamma$,
\item if $(\gamma,p) \in E_i$, then $(\gamma+j,p') \in E_i$ for all $j \in {\mathbb N_0}$ and $0 \leq p' \leq p$. 
\end{enumerate}
An index family $E = (E_1, \ldots, E_N)$ is an $N$-tuple of index sets. 
A conormal distribution $\w$ is said to be polyhomogeneous on $\mathfrak{W}$ 
with index family $E$, written as $\w\in \mathscr{A}_{\textup{phg}}^E(\mathfrak{W})$, 
if $\w$ is conormal and expands near each $H_i$ 
\[
\w \sim \sum_{(\gamma,p) \in E_i} a_{\gamma,p} \rho_i^{\gamma} (\log \rho_i)^p, \ 
\textup{as} \ \rho_i\to 0,
\]
with coefficients $a_{\gamma,p}$ being again polyhomogeneous with index $E_j$
at any $H_i\cap H_j$. 
\end{defn}

The blowup procedure amounts to a geometrically invariant way of introducing polar
coordinates around the given submanifold, so that the rays of approaching the submanifold
are distinguished in the final blowup. The blowup space is then equipped with the 
minimal differential structure with respect to which polar coordinates are smooth.
For a detailed account on the blowup procedure please consult e.g. in \cite{Mel:TAP} and \cite{Gr}. \medskip

In the special case considered here, we always treat $\sqrt{t}$ as a smooth variable
and first blow up the submanifold $A$. This defines $[M^2_h, A]$ as the disjoint union of
$M^2_h\backslash A$ with the interior spherical normal bundle of $A$ in $M^2_h$, which 
defines a new boundary hypersurface $-$ the front face ff in addition to the previous boundary faces 
$\{x=0\}, \{\wx=0\}$ and $\{t=0\}$, which lift to rf (the right face), lf (the left face) and 
tf (the temporal face), respectively.  \medskip

The actual heat-space $\mathscr{M}^2_h$ is obtained by a second blowup of  
$[M^2_h, A]$ along the diagonal $D$, lifted to a submanifold of $[M^2_h, A]$. 
As before, the lift of $D$ is cut and replaced with its spherical 
normal bundle, which introduces a new boundary face $-$ the temporal diagonal td. 
The resulting heat space $\mathscr{M}^2_h$ is illustrated in Figure 1. \medskip

\begin{figure}[h]
\begin{center}
\begin{tikzpicture}
\draw (0,0.7) -- (0,2);
\draw[dotted] (-0.1,0.7) -- (-0.1, 2.2);
\node at (-0.4,2) {t};

\draw(-0.7,-0.5) -- (-2,-1);
\draw[dotted] (-0.69,-0.38) -- (-2.05, -0.9);
\node at (-2.05, -0.6) {$x$};

\draw (0.7,-0.5) -- (2,-1);
\draw[dotted] (0.69,-0.38) -- (2.05, -0.9);
\node at (2.05, -0.6) {$\wx$};

\draw (0,0.7) .. controls (-0.5,0.6) and (-0.7,0) .. (-0.7,-0.5);
\draw (0,0.7) .. controls (0.5,0.6) and (0.7,0) .. (0.7,-0.5);
\draw (-0.7,-0.5) .. controls (-0.5,-0.6) and (-0.4,-0.7) .. (-0.3,-0.7);
\draw (0.7,-0.5) .. controls (0.5,-0.6) and (0.4,-0.7) .. (0.3,-0.7);

\draw (-0.3,-0.7) .. controls (-0.3,-0.3) and (0.3,-0.3) .. (0.3,-0.7);
\draw (-0.3,-1.4) .. controls (-0.3,-1) and (0.3,-1) .. (0.3,-1.4);

\draw (0.3,-0.7) -- (0.3,-1.4);
\draw (-0.3,-0.7) -- (-0.3,-1.4);

\node at (1.2,0.7) {\large{rf}};
\node at (-1.2,0.7) {\large{lf}};
\node at (1.1, -1.2) {\large{tf}};
\node at (-1.1, -1.2) {\large{tf}};
\node at (0, -1.7) {\large{td}};
\node at (0,0.1) {\large{ff}};
\end{tikzpicture}
\end{center}
\label{heat-incomplete}
\caption{The heat-space $\mathscr{M}^2_h$.}
\end{figure}

We now proceed with defining projective coordinates in a neighborhood of the front face in $\mathscr{M}^2_h$.
These may be used as a convenient replacement for the polar coordinates and provide explicit technical 
tools for computations on the heat space $\mathscr{M}^2_h$. The disadvantage is however that projective coordinates are 
not globally defined over the front face and one needs to choose different coordinates near each of the front face corners. 
\medskip

Near the top corner of the front face ff, projective coordinates are given by
\begin{align}\label{top-coord}
\rho=\sqrt{t}, \  \xi=\frac{x}{\rho}, \ \widetilde{\xi}=\frac{\wx}{\rho}, \ u=\frac{y-\wy}{\rho}, \ z, \ \wy, \ \wz.
\end{align}
In this coordinate system, the functions $\rho, \xi, \widetilde{\xi}$ define 
the boundary faces ff, rf and lf respectively. 
For the bottom right corner of the front face, projective coordinates are given by
\begin{align}\label{right-coord}
\tau=\frac{t}{\wx^2}, \ s=\frac{x}{\wx}, \ u=\frac{y-\wy}{\wx}, \ z, \ \wx, \ \wy, \ \widetilde{z},
\end{align}
where in these coordinates $\tau, s, \widetilde{x}$ are
the defining functions of tf, rf and ff respectively. 
For the bottom left corner of the front face, the corresponding
projective coordinates are obtained by interchanging 
the roles of $x$ and $\widetilde{x}$. Projective coordinates 
on $\mathscr{M}^2_h$ near temporal diagonal are given by 
\begin{align}\label{d-coord}
\eta=\frac{\sqrt{t}}{\wx}, \ S =\frac{(x-\wx)}{\sqrt{t}}, \ 
U= \frac{y-\wy}{\sqrt{t}}, \ Z =\frac{\wx (z-\wz)}{\sqrt{t}}, \  \wx, \ 
\wy, \ \widetilde{z}.
\end{align}
In these coordinates, tf is defined as the limit $|(S, U, Z)|\to \infty$, 
ff and td are defined by $\widetilde{x}, \eta$, respectively. 
The blow-down map $\beta: \mathscr{M}^2_h\to M^2_h$ is in 
local coordinates simply the coordinate change back to 
$(t, (x,y, z), (\widetilde{x},\wy, \widetilde{z}))$. \medskip

Asymptotic properties of the heat kernel as a polyhomogeneous
distribution on $\mathscr{M}^2_h$ have been established by the second 
author jointly with Mazzeo in \cite{MazVer:ATO} giving the following result.

\begin{thm}\label{heat-expansion}
Let $(M,g)$ be an incomplete edge space with a feasible edge metric $g$.
Then the lift $\beta^*H$ of the heat kernel is a polyhomogeneous distribution on $\mathscr{M}^2_h$ 
with the index set $(-1+m, 0)$ at ff, $(-m+\N_0, 0)$ at td, vanishing to infinite order at tf, and with a discrete 
index set $(E,0)$ at rf and lf, where $E\geq 0$.  
\end{thm}

One particular consequence of Theorem \ref{heat-expansion}
is that $H$ is square-integrable on $M\times M$ and hence the heat operator is Hilbert 
Schmidt. By the semigroup property of the heat operator, the 
heat operator is trace class and hence we conclude that the 
Laplacian $\Delta$ admits a discrete spectrum accumulating at infinity. 

\subsection{Mapping properties of the heat operator} 

Using feasibility of the edge metric, mapping properties of the heat operator were established in \cite{BV} 
with respect to certain H\"older spaces, which we now introduce.  We also refer the reader to \cite{JefLoy:RSH} and \cite{BDV} for related mapping properties.

\begin{defn} \label{holder-def}
The H\"older space\footnote{Note that in our earlier work, \cite{BV}, 
we denoted these spaces with the greek letter $\Lambda$.} 
$\ho(M\times [0,T]), \A\in (0,1),$ is defined as the space of functions 
$u(p,t)$ that are continuous on $\widetilde{M} \times [0,T]$ with finite $\A$-th H\"older norm
\begin{align*}
\|u\|_{\A}:=\|u\|_{\infty} + \sup \left(\frac{|u(p,t)-u(p',t')|}{d_M(p,p')^{\A}+
|t-t'|^{\frac{\A}{2}}}\right) <\infty. 
\end{align*}
The distance function $d_M(p,p')$ between any two points $p,p'\in \widetilde{M}$ is defined with respect to the 
incomplete feasible edge metric $g$.  In local coordinates in a singular edge neighbourhood, $d_M$ 
may be equivalently defined by
$$d_M((x,y,z), (x',y',z'))=\sqrt{|x-x'|^2+(x+x')^2|z-z'|^2 + |y-y'|^2}.$$  
\end{defn} \medskip

The higher order H\"older spaces are now defined in the following way. 
Consider the Lie algebra of edge vector fields $\mathcal{V}_e$, which are defined to be smooth 
in the interior of $\widetilde{M}$ and tangent at the boundary $\partial M$ to the fibres of the fibration. 
In local coordinates, $\mathcal{V}_e$ is then locally generated by 
\[
\left\{x\frac{\partial}{\partial x}, x\frac{\partial}{\partial y_1}, \dots, x \frac{\partial}{\partial y_b}, 
\frac{\partial}{\partial z_1},\dots, \frac{\partial}{\partial z_n}\right\}.
\]
Then the higher order H\"older spaces are defined as follows ($k\in \N_0$)
\begin{equation*}
\begin{split}
&\mathcal{C}^{1+\A}_{\textup{ie}} (M\times [0,T]) = 
\{u\in \ho \mid x^{-1}\mathcal{V}_e u \in \ho\}, \\
&\hho (M\times [0,T]) = \{u\in \ho \mid \Delta u, x^{-1}\mathcal{V}_e u, \partial_t u \in \ho\}, \\ 
&\hoonek (M\times [0,T]) = \{u\in \hoone \mid \Delta^j u \in \hoone, j\leq k \}, \\
&\hhokk (M\times [0,T]) = \{u\in \hho \mid \Delta^j u \in \hho, j\leq k\}, 
\end{split}
\end{equation*}
where differentiation is understood a priori in the distributional sense, and the H\"older norms
are given by ($\mathcal{V}_e$ is identified locally with the finite set of its generators)
\begin{equation}\label{norms}
\begin{split}
&\|u\|_{1+\A}= \|u\|_{\A} + \sum_{X\in \mathcal{V}_e} \| x^{-1}X u\|_{\A}, \\
&\|u\|_{2+\A}= \|u\|_{\A} + \|\Delta u\|_{\A} + \|\partial_t u\|_{\A} +
\sum_{X\in \mathcal{V}_e} \| x^{-1}X u\|_{\A}, \\
&\|u\|_{2k+1+\A}= \sum_{j=0}^{k} \|\Delta^j u\|_{1+\A}, 
\quad \|u\|_{2(k+1)+\A}= \sum_{j=0}^{k} \|\Delta^j u\|_{2+\A}.
\end{split}
\end{equation}
We refer the reader to \cite[Prop 3.1]{BV} for the proof that these spaces are Banach spaces.  The choice of derivatives in defining these hybrid spaces strongly depends not only on the geometry of the underlying edge space, but also on the actual equation we wish to solve. In fact such a definition seems not only to provide
a framework for treatment of parabolic Schauder-type estimates on incomplete edges, 
but has been an important tool in studying K\"ahler-Einstein edge metrics by \cite{Donaldson, JMR},
used crucially for solving the Calabi conjecture on Fano manifolds.\medskip

While the definition of the H\"older spaces above formally 
depends on a choice of background metric $g$, as a set they are invariant under certain conformal transformations of $g$.

\begin{prop} Suppose $g$ is a feasible incomplete egde metric and $g' = 
u^{\frac{4}{m-2}} g$ where $u \in \mathcal{C}^{1+\A}_{\textup{ie},g}$ is bounded away from zero, and the additional 
subscript indicates dependence on the Riemannian metric.  Then $\mathcal{C}^{k+\A}_{\textup{ie},g}
=\mathcal{C}^{k+\A}_{\textup{ie},g'}$ for $k=0,1,2$, with equivalent norms.
\end{prop}
\begin{proof}
Since $u \in \mathcal{C}^{1+\A}_{\textup{ie},g}$ is continuous up to $\partial M$, it is easy to check that 
the distances $d_M$, defined with respect to $g$ and $g'$ are equivalent. 
So if $\w \in \mathcal{C}^{\A}_{\textup{ie},g}$, 
then $\w \in \mathcal{C}^{\A}_{\textup{ie},g'}$, and the norms are equivalent.  This proves the statement for 
$k=0$. Since the derivatives in the definition of $\mathcal{C}^{1+\A}_{\textup{ie}}$ are chosen 
independent of the Riemannian metric, the statement follows for $k=1$ as well. In order to 
verify the statement for $k=2$, note that 
$$
\Delta_{g'} \w = u^{-\frac{2}{m-2}} \Delta_g \w + 
Q\{x^{-1}\mathcal{V}_e u, x^{-1}\mathcal{V}_e \, \w, u\},
$$ 
where $Q$ denotes some smooth
combination of the elements in the brackets. \end{proof}

Let us now recall the fundamental result in \cite{BV} on the mapping properties of the heat operator
with respect to these H\"older spaces. The result of our previous work in \cite[Theorem 3.2]{BV}
has been established under stronger assumptions on the metric $\ginit$ and in fact a small
change of argument, as outlined in the work of the second named author in \cite{Ricci-Vertman},
is needed to extend the statement to our setting of feasible edge metrics. 

\begin{thm}\label{mapping} 
Let $(M^m,g)$ be a feasible incomplete edge space. Then for $\A \in (0,1)$ sufficiently
small and any $k\in \N_0$, the heat operator $e^{t\Delta}$ acts as a bounded convolution operator 
\begin{align*}
&e^{t\Delta}: \hhok (M\times [0,T]) \to 
\left( \mathcal{C}^{2(k+1)+\A}_{\textup{ie}} \cap \sqrt{t} \, \hhok \right) (M\times [0,T]), \\
&e^{t\Delta}: \mathcal{C}^{2k+1+\A}_{\textup{ie}} (M\times [0,T]) \to 
\sqrt{t} \, \mathcal{C}^{2(k+1)+\A}_{\textup{ie}} (M\times [0,T]).
\end{align*}
\end{thm}

\begin{proof}
The second mapping property follows from \cite{BV} for our setting of 
feasible edge metrics ad verbatim without further arguments. For the first
statement we only provide a brief sketch of the argument for $k=0$. 
Since in contrast to \cite{BV} we do not impose the condition of \cite[Definition 1.3 (iii)]{BV} here, we need to 
point out how this assumption is actually used in the proof of \cite[Theorem 3.2]{BV}.
That condition has been employed in order to refine
the statement on the asymptotic behaviour of the heat kernel at the right boundary 
face of $\mathscr{M}^2_h$. \cite[Proposition 2.4]{BV} asserts that the coefficients
of $\rho^0_\rf$ and $\rho^2_\rf$ in the heat kernel expansion at rf is are fact harmonic functions on $F$.
Without the assumption \cite[Definition 1.3 (iii)]{BV}, the asymptotic 
expansion of the heat kernel established in \cite[Proposition 2.4]{BV} still holds, however
the coefficient of $\rho^2_\rf$ need not be harmonic on $F$ anymore.\medskip
 
In order to make our argument precise, let us introduce the following notation
for the H\"older spaces, which we will use in this proof only. We want to distinguish between 
second order H\"older spaces where we do or do not require regularity under an application of the Laplace
Beltrami operator by an additional subscript. We write
 \begin{equation*}
\begin{split}
&\hhod (M\times [0,T]) = \{u\in \ho \mid \{ x^{-1}\mathcal{V}_e^2, 
x^{-1}\mathcal{V}_e, \sqrt{t} \partial_t, \textup{id}\} \, u \in \ho \}, \\
&\hho (M\times [0,T]) = \{u\in \hhod \mid \Delta u, \partial_t u  \in \ho \}.
\end{split}
\end{equation*}
An examination of the argument in the proof of \cite[Theorem 3.2]{BV} shows that 
even if the coefficient of $\rho^2_\rf$ is not harmonic on $F$, the following mapping 
properties still follow ad verbatim
\begin{align*}
e^{t\Delta}: \ho (M\times [0,T]) \to 
\left( \hhod \cap \sqrt{t} \, \ho \right) (M\times [0,T]).
\end{align*}
In fact the estimates may be considerably simplified using 
\cite[Corollary 3.2]{Ricci-Vertman}. Only the mapping of $\ho$ into $\hhod$ requires some care 
and as written down in \cite{BV} uses the assumption \cite[Definition 1.3 (i) and (iii)]{BV} decisively. 
Without that assumption we may proceed as follows. We write for any $u \in \ho$ using the heat equation
\begin{align*}
\Delta e^{t\Delta} u = \partial_t e^{t\Delta} u - u.
\end{align*}
This allows us to avoid discussing the intricate right face asymptotics of the heat kernel to the order of $\rho_\rf^2$, 
since the lift of $\partial_t$ to the heat space $\mathscr{M}^2_h$ applied to the heat kernel, lowers its front face
behaviour by an order of two, but does not affect the expansion of the heat kernel at rf. 
Now, we may estimate the H\"older norm of $\partial_t \, e^{t\Delta} u$ exactly as before in 
\cite[Theorem 3.2]{BV}, and conclude that $e^{t\Delta}: \ho \to \hho$ is indeed bounded.
Thus, \cite[Theorem 3.2]{BV} is still satisfied under the milder feasibility assumptions imposed here.
Extension to general $k\in \N_0$ is straightforward by uniqueness of solutions to the heat equation.
\end{proof}

The present analysis actually requires some additional mapping property
of the heat operator, which are proved along the lines of Theorem \ref{mapping}.
It concerns the action of the heat 
operator as a convolution operator on the space $\mathscr{B}$ of bounded functions on $\widetilde{M}$.
If we discuss $\mathcal{C}^{1+\A}_{\textup{ie}}$ regularity of $e^{t\Delta} \mathscr{B}$, 
following the estimates in \cite{BV}, we apply one less derivative to the heat kernel and hence
are left with additional $\rho_\ff \rho_\td$ in the estimates, which we can easily convert into 
$d_M^\A$ factors. Consequently, the heat kernel admits following mapping properties in addition
to the one established in \cite{BV} as stated in Theorem \ref{mapping}.

\begin{prop}\label{mapping2}
Denote by $\mathscr{B}$ the space of bounded functions on $\widetilde{M}$. Then
the heat operator acting by convolution in time is a bounded mapping
\begin{equation*}
\begin{split}
e^{t\Delta}: 
\mathscr{B} \to \mathcal{C}^{1+\A}_{\textup{ie}}. 
\end{split}
\end{equation*}
\end{prop}

\subsection{Invariance of the edge structure under $\hho$ 
conformal transformations}\label{invariance} 

An arbitrary conformal transformation of an incomplete edge metric can destroy the edge structure we hope to preserve.  Fortunately, as we now prove, conformal transformations in $\hho$ preserve the edge structure.

\begin{lemma}  Suppose $g$ is a feasible incomplete edge metric, and $u$ is a positive function lying in $\hho(M)$.  Then $u^{\frac{4}{m-2}} g$ is an incomplete edge metric.
\end{lemma}
\begin{proof}
Since $u \in \hho$ we may apply the mean value theorem. Due to the fact 
that an element of $\ho$ must be independent of $z$ at $x=0$, we obtain 
an expansion as $x\to 0$
\[ u(x,y,z) = u_0(y) + O(x). \]
Now we substitute $\xt = u_0^{\frac{2}{m-2}} x$ and can expand as $x\to 0$
\begin{align*}
u^{\frac{4}{m-2}} (dx^2 + \phi^*g^B +  x^2 g^F) &= u_0^{\frac{4}{m-2}} dx^2 + 
\phi^* (u_0^{\frac{4}{m-2}} g^B) + u_0^{\frac{4}{m-2}} x^2 g^F + O(x)
\\&= d\xt^2 + \phi^*( u_0^{\frac{4}{m-2}} g^B ) + \xt^2 g^F + O(x).
\end{align*}
The key point here is that up to a conformal transformation of the base metric, 
the leading term of the metric has the same rigid edge structure. 
\end{proof}

Note that this lemma does not assert that $g' = u^{\frac{4}{m-2}} g$ remains feasible.  
At present we do not know if conditions \textit{(i)} of Definition \ref{def-feasible} hold for $g'$.

\subsection{H\"older regularity of the initial scalar curvature} 

We conclude this introductory section with a few remarks about one of the hypothesis of Theorem \ref{BV}.  

There is a natural obstruction for the existence of a constant scalar curvature metric with bounded conformal factor on an edge manifold.  To see this, assume for simplicity $b=0$ and consider the short-time solution $g=e^{2v} \ginit$
of the edge Yamabe flow, as obtained in \cite{BV} on a (finite) time interval $[0,T)$. For any $\lambda >0$, consider the rescaled metric $\lambda^2 g$, 
which is again a solution of the edge Yamabe flow. Denoting local coordinates on $F$ by $z$, near the singularity
$$
\exp \left(2v\left(\frac{t}{\lambda^2}, \frac{x}{\lambda}, z\right)\right)\left(dx^2 + x^2 g^F\right)
$$
solves the Yamabe flow equation for $t\in [0,\lambda T)$ and $(x,z)\in (0,\lambda) \times F$.
Taking $\lambda$ to infinity, we obtain a solution $e^{u(0,0,z)} (dx^2 + x^2 g^F)$ on the 
infinite cone $(0,\infty) \times N$ for all times $t>0$. Since it is time-independent, its 
scalar curvature is zero by the unrescaled Yamabe flow equation. Continuity of $u$ up to the singularity
implies that $u(0,0,z)$ is in fact a constant. Hence 
$$
\textup{scal} \left(dx^2 + x^2 g^F\right) = \frac{\textup{scal} (g^F) - n(n-1)}{x^2} = 0.
$$
So short-time existence of the edge Yamabe flow with conformal 
factors continuous up to the singularity already requires that the scalar 
curvature of the fibres $(F,g^F)$ must be given by $n(n-1)$. \medskip

Thus, in addition to feasibility of the edge metric, we also require 
$$
\scal(\ginit) \in \hho (M),
$$
and consequently, the scalar curvature of the fibres $F$ is $n(n-1)$. \medskip

\section{Maximum principles for spaces with incomplete edges} \label{sec:maxprin}

Our long-time existence argument rests strongly on maximum principles adapted to incomplete edges.  
As pointed out by Jeffres in the case of incomplete conic metrics \cite{Jeffres}, one must be careful about 
the possibility that extrema of solutions to the heat equation occur at the conic points.  
As we will explain, we may rule this out since we only consider solutions in our H\"older spaces.  \medskip

We begin with an adaptation of the classical principle.  This maximum principle was developed 
jointly with Mazzeo. We always denote the Laplace Beltrami operator of a feasible incomplete edge metric $g$ by $\Delta$.

\begin{theorem} \label{maxprin} Suppose $u \in \hho$ attains its minimum (respectively maximum) at 
some point $p\in \widetilde{M}$. In particular, we do not require $p$ to lie in the interior of the edge space. 
Then $(\Delta u)(p) \geq 0$ $($respectively $(\Delta u)(p) \leq 0 )$.
\end{theorem}

\begin{proof} If $p$ is an interior point, then this result is classical.  Otherwise $p$ lies on the edge, 
and in coordinates $(x,y,z)$ near the edge we may assume $p = (0,0,0)$.  We now outline the 
argument if $g$ is a rigid conic metric of the form
\[ g|_{\U} = dx^2 + x^2 g^F, \]
and then we describe how to adapt the proof in the case of general conic metrics and general 
incomplete edges.  Thus we begin by assuming $p$ is a conic point. \medskip

Recall that since $u \in \hho$ we have that $x^{-1}\mathcal{V}_e u$ 
and $\Delta u$ lie in $\ho$, and in particular these 
derivatives extend continuously to the boundary $\partial M$ 
of the resolution $\widetilde{M}$ with their restrictions
to $x=0$ being constant. In case of conical singularities 
this means that $x^{-1}\mathcal{V}_e u$ and $\Delta u$ are continuous 
and constant at $p$. \medskip

Suppose that $p$ is a minimum for $u$.  If by way of contradiction $\Delta u(p) < 0$, then 
by continuity there exists $\varepsilon > 0$ sufficiently small such that $\Delta u < 0$ on 
$B_\varepsilon=\{ x \leq \varepsilon \}$.  
It has been observed in \cite{BV} that $\hho$ lies inside the domain for the Friedrichs
self-adjoint extension of $\Delta$. Hence for any $u_1,u_2\in \hho$ integration by parts does not yield
boundary terms coming from the singularity (recall, $n = \dim F$)
$$
\langle \Delta u_1, u_2\rangle_{L^2(B_\varepsilon)} = \langle u_1, \Delta u_2\rangle_{L^2(B_\varepsilon)} + 
\int_F \varepsilon^n ((\partial_x u_1) u_2 - u_1 (\partial_x u_2))|_{x=\varepsilon} \, \textup{dvol} (g^F).
$$
Since $u$ and $1$ both lie in $\hho$, integrating by parts we find 
\begin{equation} \label{int-by-parts}
 0 > \int_{B_\varepsilon} \Delta u \; \dv = \int_F 
 \varepsilon^n (\partial_x u)|_{x=\varepsilon} \, \textup{dvol} (g^F).
\end{equation}
Denote the rescaled cross section $\{ x = \varepsilon \}\equiv (F, \varepsilon^2 g^F)$ by $A_\varepsilon$. 
The volume form of $A_\varepsilon$ is given by $\varepsilon^n \textup{dvol} (g^F)$.
Now consider the average value, $\ubar(\varepsilon)$, of $u$ over $A_\varepsilon$,
\[ \ubar(\varepsilon) = \frac{1}{\textup{vol} (A_\varepsilon)} \int_{A_\varepsilon} u(\varepsilon, z) dA_{\varepsilon} = 
\frac{1}{\textup{vol} (A_1)} \int_{A_1} u(\varepsilon, z) dA_1. \]
Since $u$ is continuous up to $p$ and constant at $p$, clearly, $\lim_{\varepsilon \to 0^+} \ubar(\varepsilon) = u(p)$.
Since $p$ is a minimum of $u$, we deduce that
$\ubar(\varepsilon) \geq u(p)$ for all $\varepsilon>0$ sufficiently small.  Consequently, the (right) 
derivative $\left. \frac{d}{d \varepsilon}\right|_{\vep = 0^+} \ubar \geq 0$.  However,
\begin{equation} \label{diff-ubar}
\frac{d}{d \varepsilon} \ubar(\varepsilon) = \frac{1}{\textup{vol} (A_1)} \int_{A_1} \partial_x u(\varepsilon, z) dA_1 < 0, 
\end{equation}
for all $\varepsilon$ sufficiently small, by the inequality \eqref{int-by-parts}.  But this 
implies $\ubar$ is decreasing, which is a contradiction. So we conclude $(\Delta u)(p) \geq 0$.
\medskip

When $g$ is no longer rigid but merely a feasible conic metric $g=g_0+ e$, 
with the leading term $g_0$ being rigid, $g_0|_{\U} = dx^2 + x^2 g^F$, and 
the higher order term $e$ being a symmetric tensor on $M$ that satisfies $|e|_{g_0} = O(x^2)$,
we may still write 
$$
g = dx^2 + x^2 h(x,z;dx,dz),
$$ 
and adapt the argument above as follows.
We will again restrict to $A_\varepsilon$ level sets.  
Note that $\partial_x$ is not quite a unit normal for this hypersurface, in fact, 
$|\partial_x|_{g} = 1 + O(\varepsilon^2)$. In what follows, we relabel $h$ to include 
these error terms. We set with respect to the relabeled tensor $h$, we write
\[ 
d A_{\varepsilon} := \varepsilon^{n} \sqrt{ \det(h(\varepsilon,z)) } \, \textup{dvol} (g^F) = \varepsilon^{n} 
\sqrt{\frac{\det(h(\varepsilon,z))}{\det(h(1,z))}} \, dA_1,
\]
so that upon integration by parts as in equation \eqref{int-by-parts} we find again
\begin{equation} \label{int-by-parts2} 
0 > \int_{A_\varepsilon} (\partial_x u)(\varepsilon, z) \; 
d A_{\varepsilon} = \int_{\{x = 1\}} (\partial_x u) (\varepsilon, z) \cdot \varepsilon^{n} 
\sqrt{\frac{\det(h(\varepsilon,z))}{\det(h(1,z))}} \, dA_1.
\end{equation}
Now consider the average value of $u$ once more
\begin{align*}
\ubar(\varepsilon) := \frac{1}{\int_{A_1} \varepsilon^{n} \sqrt{ \frac{\det(h(\varepsilon,z))}{\det(h(1,z))}} \, dA_1 } 
\int_{A_1} u(\varepsilon, z) \varepsilon^{n} \sqrt{\frac{\det(h(\varepsilon,z))}{\det(h(1,z))}} \, dA_1 \\
= \frac{1}{\int_{A_1} \sqrt{\frac{\det(h(\varepsilon,z))}{\det(h(1,z))}} \, dA_1 } 
\int_{A_1} u(\varepsilon, z) \sqrt{\frac{\det(h(\varepsilon,z))}{\det(h(1,z))}} \, dA_1. 
\end{align*}
As before we wish to differentiate $\ubar(\varepsilon)$ in $\varepsilon$ and evaluate the derivative
at zero. Note that the dependence on $\varepsilon$ is now more complicated.  We write $q(\varepsilon) 
:= \sqrt{\frac{\det(h(\varepsilon,z))}{\det(h(1,z))}}$.  Given an expansion for $h$, 
we may write $h(\varepsilon, z) = h_0(z) + h_1(z) \varepsilon + O(\varepsilon^2)$, as $\varepsilon \to 0^+$, 
where $h_0(z), h_1(z)$ are smooth $2$-tensors. Then
\[ q(\varepsilon) := 1 + H_1 \varepsilon + O(\varepsilon^2), \ \textup{as } \, \vep \to 0^+. \]
This expansion is differentiable in $\varepsilon$ and hence we find as $\vep \to 0^+$
\begin{align*}
\frac{d}{d\varepsilon} \ubar(\varepsilon) &= \frac{1}{\int_{A_1} q \, dA_1 } \int_{A_1} (\partial_x u)(\varepsilon, z) q \, dA_1 
\\ &+ \frac{1}{\int_{A_1} q \, dA_1 } \int_{A_1} u(\varepsilon, z) (H_1 + O(\varepsilon) ) \, dA_1 -\ubar(\varepsilon) \, \frac{\int_{A_1} (H_1 + O(\varepsilon)) \, dA_1 }{\int_{A_1} q \, dA_1 }  \\
&= \frac{1}{\int_{A_1} q \, dA_1 } \int_{A_1} (\partial_x u)(\varepsilon, z) q \, dA_1 + \frac{1}{\int_{A_1} q \, dA_1 } \left\{ \int_{A_1} \left( u(\varepsilon,z) - \ubar(\varepsilon) \right)  H_1(z) \, dA_1 \right\} 
\\ &\hspace{12.5cm}+ O(\varepsilon).
\end{align*}
Note that as $\varepsilon \to 0^+$, the second integral approaches zero by dominated convergence, 
and so applying \eqref{int-by-parts2} we find that $\left. \frac{d}{d \varepsilon}\right|_{\vep =0^+} \ubar \geq 0$ 
as before, yielding the contradiction. \medskip

Finally we consider the general case of an incomplete edge which is a bundle of cones over 
the base manifold $B$.  Assume the minimum $p$ of $u$ occurs at a point on the edge 
which we label $(0,0,0)$ in $(x,y,z)$ coordinates.  Introducing polar coordinates 
from $p$ in $x$ and $y$ allows us to write up to higher order terms
\[ g = ds^2 + s^2 h(\omega, z; d\omega, dz), \]
where $\omega$ lies on the unit sphere $\mathbb{S}^b$.  
This now reduces the edge case to the conic case.
\end{proof}

We conclude the section by formulating some parabolic maximum principles.  
The first principle is classical, see for example \cite{Evans}.

\begin{lemma}[Classical maximum principle]
Suppose that $u \in \hho(M \times [0,T])$ satisfies
\[ \partial_t u \leq \Delta u, u(p,0) = 0. \]
Then $\max_{\widetilde{M} \times [0,T]} u$ occurs on the parabolic boundary of $\widetilde{M} \times (0,T]$,
which is a union of $\{x=0\} \times (0,T)$ and $\widetilde{M} \times \{t = 0\}$.
\end{lemma}
\begin{proof}
Suppose by way of contradiction that the maximum occurs somewhere in $\widetilde{M} \times (0,T]$, 
away from the edge $\{x=0\}$.  Let $p_0 = (x_0, y_0, z_0, t_0)$, $x_0 > 0$ be a point where 
the maximum of $u$ is attained.  We may reduce to the classical maximum principle technique.  
We consider two cases.  If $\partial_t u < \Delta u$ (strict inequality), then evaluating this expression at $p_0$ gives
\[ 0 \leq \partial_t u(p_0)- \Delta u(p_0)< 0, \]
where the first inequality follows from $\partial_t u(p_0) \geq 0$ and $\Delta u(p_0) \leq 0$ at a maximum point.  
This is a contradiction and, consequently, the maximum occurs on the parabolic boundary of $\widetilde{M} \times (0,T]$.
In the second case of $\partial_t u \leq \Delta u$, set $u_{\varepsilon} = u - t\varepsilon$.  Now
\[ \partial_t u_{\varepsilon} = \partial_t u - \varepsilon \leq \Delta u_{\varepsilon} - \varepsilon < 0, \]
and so by the previous argument, the maximum of $u_{\varepsilon}$ occurs on the parabolic 
boundary.  We now let $\varepsilon \to 0$.
\end{proof}

As a corollary we also have a parabolic maximum principle. 

\begin{cor} \label{JeffresMaxPrinciple}
Suppose $u \in \hho(M \times [0,T])$ satisfies
\[ \partial_t u < \Delta u, u(p,0) = 0. \]
Then $u \leq 0$ on $M \times [0,T]$.
\end{cor}

\begin{proof}
By the previous lemma, the maximum of $u$ occurs at the 
parabolic boundary of $\widetilde{M} \times (0,T]$,
which consists of $\{x=0\} \times (0,T)$ and $\widetilde{M} \times \{t = 0\}$.
If the maximum lies on $\widetilde{M} \times \{t = 0\}$, the statement follows
from the initial condition $u(p,0) = 0$. If the maximum lies on $\{x=0\} \times (0,T)$, 
then by Theorem \ref{maxprin} we have $\partial_t u < 0$ at the maximum point. 
This is a contradiction.
\end{proof}

Note that a straightforward adaption of these maximum principles allows us to handle negative zeroth order terms.
\begin{cor} \label{PMP2} Let $c \in \hospace(M \times[0,T])$ be a negative function.
Suppose $u \in \hho(M \times [0,T])$ satisfies
\[ \partial_t u < \Delta u + cu, u(p,0) = 0. \]
Then $u \leq 0$ on $M \times [0,T]$.
\end{cor}

We conclude this section briefly comparing the maximum principle given here and the 
one developed in \cite{Jeffres}.  Using suitable barrier functions as introduced by 
Jeffres, one may prove the corollary above in the conic case without the restriction 
to $\hho$.  The authors spent considerable time trying to adapt 
Jeffres's techniques to obtain $L^{\infty}$ estimates for the Yamabe flow.  We were 
unable to discover how to use these barrier functions to obtain $L^{\infty}$ 
estimates using the differential inequality technique given in \cite{Ye}.

\section{Parabolic regularity of a heat-type equation} \label{parabolic-section}

The purpose of this section is to prove H\"older regularity for solutions of a certain quasilinear parabolic equation that arises in the study of the Yamabe flow.  To put this in context, recall that in \cite{BV} we proved parabolic Schauder-type estimates in modified H\"older spaces for a linear heat equation of the form
\begin{equation} \label{eqn:heat-equation}
\left\{ \begin{array}{ll} \partial_{t} u - \Delta u &= f, \\
                            u(x,0) &= u_0(x). \end{array} \right. 
\end{equation}
In classical parabolic PDE theory, one expect such estimates to hold for uniformly parabolic operators with H\"older coefficients.  As it does not appear to be immediate to pass from estimates of   equation \eqref{eqn:heat-equation} to more general equations with time-dependent coefficients, we state and prove the following more general existence and regularity result. Our work is modeled after the very detailed argument given in \cite{EM}.

\begin{prop} \label{prop:mainreg}
Let $a(p,t)$ be a positive function bounded from below that lies in $\hoone(M \times [0,T])$, 
and consider the operator $P = \partial_t + a \Delta$.  Then there exists a right parametrix $Q$ for $P$ such that
\[ Q: \ho(M \times [0,T]) \to \hho(M \times [0,T]) \]
is a bounded map and if $f \in \ho(M \times [0,T])$, then $u = Qf$ is a solution to the equation
\[ (\partial_t + a \Delta) u = f, \quad u(p,0) = 0. \]
\end{prop}

The proof of this proposition uses the method of frozen coefficients after suitable localization in a neighborhood of a boundary point in conjunction with parabolic Schauder estimates.  We provide a fairly complete detail of this construction, but as the argument seems to be standard the reader  interested only in the estimate specific to incomplete edges may consult Lemma \ref{lemma:main-frozen-est}.  Sections \ref{sec:reg-prelim}--\ref{sec:reg-parametrix} are self-contained and are not cited in the remainder of the paper. 

\subsection{Preliminaries to the global analysis} \label{sec:reg-prelim}

In this discussion, as in the introduction, we will use $M$ for the 
original incomplete edge manifold and $\widetilde{M}$ for its resolution. 
By a slight abuse of notation we denote the boundary of $\widetilde{M}$ by $\partial M$.
We assume a preliminary diffeomorphism has been performed so 
that $x$ is the radial distance to the boundary and we have a  collar 
neighborhood of the boundary denoted by $U_{R} = \{ x \leq R \}$. 
\medskip 

We will need to localize our argument.  To do this, we introduce a 
special covering of $\widetilde{M}$.  First we introduce a model 
half-cube centered at the origin in coordinates $(x,y,z)$
\begin{align*}
C(r) = [0,r) \times (-r,r)^b \times (-r,r)^n \subset [0,\infty)_x \times 
\bR^b_y \times \bR^n_z. 
\end{align*}
Given a point $p \in \partial M$, we may find a coordinate chart mapping $C(1)$
to some open subset $W_p \subset \widetilde{M}$ of $p$ of the form $\phi_p: C(1) \rightarrow W_p$.
Compactness of the boundary implies that we may find finitely many such charts 
$\{W_i, p_i, \phi_i\}_{i=1}^N$ that will cover a collar neighborhood of the form 
$U_{R}$ for $R$ sufficiently small.  We add the open subset 
$W_0 = \widetilde{M} \backslash \{ x \leq R/2 \}$ to this covering for a 
cover of $\widetilde{M}$ that we refer to hereafter as the \emph{reference covering}.
\medskip

The H\"older norms in $\ho$ and $\hho$ may be equivalently defined
using a partition of unity $\{\chi_i\}$ subordinate to the reference covering $\{W_i\}$. 
More precisely, taking local H\"older norms $\|(\chi_i f)\circ \phi_i \|_{\hospace(C(1))}$
for any function $f: W_i \to \bR$, we can define a global H\"older norm
\[ \| f \|_{\hospace(M)}^{(W_i,\phi_i)} = \sum_{i=0}^N \| (\chi_i f)\circ \phi_i \|_{\hospace(C(1))}, \]
with the obvious extensions to higher-order H\"older spaces. These norms are equivalent
to the H\"older norms defined in \eqref{norms} and moreover, given any other
choice of reference covering and subordinate 
partition of unity ${(W'_i, p'_i, \phi'_i)}$, the H\"older norms  defined with respect to 
$\{W_i, p_i, \phi_i\}$ and ${(W_i', p_i', \phi_i')}$ are equivalent as well. 
We fix a reference cover for the remainder of the section. \medskip

In constructing a boundary parametrix, we will need to localize in small domains near the edge, 
which requires some care, since the arbitrary cutoff functions in $\widetilde{M}$ need not 
need not respect the conic identification at $x=0$ and hence need not define functions on 
the edge manifold $M$ (or equivalently in the incomplete edge H\"older spaces).  Let $\sigma: \bR^+_s \to [0,1]$ 
($s$ is the coordinate on $\R^+$) be a smooth cutoff function with $\sigma \equiv 1$ for $|s| \leq 1/2$ and $\sigma \equiv 0$ for $s \geq 1$.  For any point $p = (0,0,0)$ in local coordinates $(x,y,z)$ at the edge, let $s$ denote polar coordinates in $x$ and $y$ from $p$.  Now set $\phi(x,y,z) = \sigma(x)\sigma(\|y\|)$ and $\psi(x,y,z) = \sigma(\frac{x}{2})\sigma(\frac{\|y\|}{2})$.  Now $\phi, \psi \in C^\infty(\widetilde{M})$, are both identically $1$ near $p$ and moreover 
$\psi \equiv 1$ on $\supp \phi$. These functions pass to the quotient 
on the fibers and hence are well-defined on
$\overline{M}$ as well.  In particular $\phi, \psi \in \hho(M)$ since these functions are constant near the edge. For any point 
$p = (0, y_0, z_0) \in \partial M$ in coordinates of a reference chart $W_j$, 
define for some $\vep \in (0,1)$ to be specified
\begin{align*} 
\widetilde{\varphi}_{j,p} = \varphi\left( \frac{x}{\vep}, \frac{y-y_0}{\vep}, z-z_0 \right),
\quad \widetilde{\psi}_{j,p} = \psi\left( \frac{x}{\vep}, \frac{y-y_0}{\vep}, z-z_0 \right).
\end{align*}

Now consider for any $j=1,...,N$ and any $\vep \in (0,1)$ the lattice 
$L_{j,\vep}$ of points in the coordinate chart $W_j$ of the form 
$\{ (0,\vep w) \mid w \in \bZ^{b+n}\}$.  Every point on $\partial M$
lies in the support of at most a fixed number of functions 
$\{ \psi_{j,p}: j = 1,\cdots, N, p \in L_{j,\vep}\}$.  Set
\begin{align*} 
\varphi_{j,p} = \frac{ \widetilde{\varphi}_{j,p}}{ \sum_k \sum_{q \in L_{j,\vep}} \widetilde{\varphi}_{k,q}},
\quad \psi_{j,p} = \frac{ \widetilde{\psi}_{j,p}}{ \sum_k \sum_{q \in L_{j,\vep}} \widetilde{\psi}_{k,q}}.
\end{align*}
Now $\varphi_{j,p}$ and $\psi_{j,p}$ are partitions of unity for any choice of $\vep$, and the functions
\begin{align*} 
\Phi := \sum_j \sum_{p \in L_{j,\vep}} \varphi_{j,p} \equiv 1, 
\quad \Psi := \sum_j \sum_{p \in L_{j,\vep}} \psi_{j,p} \equiv 1,
\end{align*}
are identically $1$ in a neighborhood of the edge.  

\subsubsection{An auxilary lemma}

\begin{lemma} \label{lemma:aux-smoothing}
Suppose that $\psi, \phi \in \hospace(M)$ are smooth functions with compact support,
and $\psi$ is compactly supported away from the singular edge neighborhood.  
Let $H$ be the heat operator of the Laplace Beltrami operator.  Then
\[ E_0 := \psi H \phi: \mathcal{C}^{k+\A}_{\textup{ie}}(M \times [0,T]) \rightarrow 
\mathcal{C}^{k+2+\A}_{\textup{ie}}(M \times [0,T]), \]
where $\phi$ and $\psi$ act by multiplication
and the operator norm $\| E_0 \| \rightarrow 0$ as $T \rightarrow 0^+$.
\end{lemma}
\begin{proof} 
The kernel of $E_0$ is not stochastically complete in the sense of ((3.1), \cite{BV}), 
hence the arguments of (Theorem 3.2, \cite{BV}) do not apply here. 
However, since $\psi$ is compactly supported away from the singular edge neighborhood, 
the lift of the kernel $\psi H \phi$ to the blown up heat space $\mathscr{M}^2_h$ is 
compactly supported away from the
front and the right boundary faces. Hence we find
$$
\beta^* (\psi H \phi) = (\rho_\ff \rho_\rf \rho_\tf)^\infty \rho_\td^{-m} G,
$$
with $G$ bounded and vanishing to infinite order as $|(S,U,Z)|\to \infty$. 
This allows us to establish the stated mapping properties by straightforward estimates in local 
projective coordinates near the various corners of the front face in $\mathscr{M}^2_h$ 
along the lines of the estimates for $I_3$ in \cite{BV}. Moreover, $(\rho_\ff \rho_\td)= O(\sqrt{t})$
as $t\to 0$, so that in particular $\| E_0 \| \rightarrow 0$ as $T \rightarrow 0^+$.
Basically, this lemma follows by the classical parabolic Schauder estimates and does not
depend on the specific singular structure of the manifold. 
\end{proof}

\subsection{Construction of a boundary parametrix}

Let $a \in \hoone(M \times [0,T])$, and $f \in \ho(M \times [0,T])$.  
Given the setup above, we may regard any of the functions 
$\varphi_{j,p} f$ as lying in $\ho(M \times [0,T])$ by extending $f$ 
to be zero outside of the support of the $\varphi_{j,p}$.  
We may freeze the coefficient of $a(x,y,z,t)$ at $(p,0)$, and consider the equation
\begin{align*}
(\partial_t + a(p,0) \Delta) \widetilde{u_p} &= \varphi_{j,p} f, \\
\widetilde{u_p}(t = 0) = 0.
\end{align*}
A solution to that initial value problem is given by $\widetilde{u_p} = H_p [ \varphi_{j,p} f ]$, 
where $H_p$ denotes the heat kernel corresponding to the frozen coefficient equation. 
Our previous work \cite{BV} implies after a simple time rescaling that 
$\widetilde{u_p} \in \hho(M \times [0,T])$.  Set $u_p = \psi_{j,p}  
H_p [ \varphi_{j,p} f ]$ and define an initial approximation to a boundary parametrix by
\[ Q_b f = \sum_{j=1}^N \sum_{p \in L_{j,\vep}} \psi_{j,p} H_p [ \varphi_{j,p} f ]. \]
Before we proceed with studying $Q_b$, we make the following observation.

\begin{lemma} \label{lemma:main-frozen-est}
The solution $u_p = \psi_{j,p}  H_p [ \varphi_{j,p} f ]$ satisfies
\begin{align*}
(\partial_t + a \Delta) u_p &= \varphi_{j,p} f + E_{j,p}^0 f + E_{j,p}^1 f,
\end{align*}
where $a\in \hoone(M \times [0,T])$ is the variable coefficient factor and
\begin{enumerate}
\item $E_{j,p}^0: \ho(M \times [0,T]) \to \ho(M \times [0,T])$ is bounded and 
there exists a constant $C>0$ independent of $j, p$ where if $T/\vep^2 < 1$, then
$\|E^0_{j,p} f\|_{\alpha} \leq C (\vep + T^{\alpha/2}) \| \varphi_{j,p} f \|_{\alpha} $,
\item $E_{j,p}^1: \ho(M \times [0,T])\to \ho(M \times [0,T])$ is bounded with 
operator norm satisfying $\lim_{T \to 0} \| E^1\| = 0$. 
\end{enumerate}
\end{lemma}

\begin{proof}
We apply the variable coefficient operator to $u_p$ and estimate.  
Note that we drop the subscripts on $\varphi, \psi$ and the error terms $E^i$ 
for simplicity.  Furthermore, all norms and semi-norms will be local and so we simply 
indicate the H\"older index.  We remind the reader that constants may change from line to line.
We compute 
\begin{equation}\label{eqn:var-est}
\begin{split}
(\partial_t + a \Delta) u_p &= (\partial_t + a \Delta) \left( \psi  H_p [ \varphi f ]\right) \\
&= \psi (\partial_t + a \Delta) \left(   H_p [ \varphi f ]\right) + [\psi, a\Delta] H_p [ \varphi f ]  \\
&= \psi (\partial_t + a(p,0) \Delta) \left(   H_p [ \varphi f ]\right) 
\\ &+ \psi ( (a - a(p,0)) \Delta) \left(   H_p [ \varphi f ]\right) + [\psi, a\Delta] H_p [ \varphi f ]\\
&= \psi \varphi f  + \psi ( (a - a(p,0)) \Delta) \left(   H_p [ \varphi f ]\right) + [\psi, a\Delta] H_p [ \varphi f ],
\end{split}
\end{equation}
where we have interpolated the constant coefficient $a(p,0)$ above.    Define
\begin{equation}\label{E-1-2}
\begin{split}
E^0 f &:= \psi ( (a - a(p,0)) \Delta) \left(   H_p [ \varphi f ]\right), \; \mbox{and}\\
E^1 f &:= [\psi, a\Delta] H_p [ \varphi f ].
\end{split}
\end{equation}
For the first term of equation \eqref{eqn:var-est}, note that since 
$\psi \equiv 1$ on the support of $\varphi$, that $\psi \varphi f = \varphi f$.
For the third term of equation \eqref{eqn:var-est}, note that for any function $v$,
\[ [\psi, a\Delta] v = -2a \inn{\nabla \psi}{\nabla v} - a v \Delta \psi. \]
Since $\psi \equiv 1$ on the support of $\varphi$, any derivative of $\psi$ 
vanishes on $\supp \varphi$, and consequently near the boundary.  
Thus by Lemma \ref{lemma:aux-smoothing}, $E^1: \ho(M \times [0,T])\to 
\ho(M \times [0,T])$ is bounded with the operator norm of $E^1$ 
tending to zero as $T \to 0$. \medskip

It remains to establish mapping properties for $E^0$ and estimate 
its operator norm. First, observe that $\varphi f \in \ho$, and so by earlier 
Schauder estimates, $H_p [\varphi f] \in \hho$ with 
\begin{align*}
 \| H_p [\varphi f] \|_{2+\alpha} \leq C \| \varphi f\|_{\alpha}, 
\end{align*}
for a constant $C$ remaining bounded as $T \to 0$.  Thus,
\begin{align}\label{Hf}
 \| \Delta H_p [\varphi f] \|_{\alpha} \leq C' \| \varphi f\|_{\alpha}. 
\end{align}
For the remainder of the proof we introduce the following two abbreviations.
We write $w := \Delta H_p [\varphi f]$ and also for any $v \in \ho$
$$
[v]_\A := \sup \frac{|v(p,t)- v(p',t')|}{d_M(p,p')^\A+|t-t'|^{\frac{\A}{2}}}.
$$
We can now estimate
\begin{equation} \label{eqn:est-key-term-first}
\begin{split}
 \| \psi ( a - a(p,0) ) w\|_{\alpha} &= \| \psi ( a - a(p,0)) w\|_{\infty} + [ \psi(a-a(p,0)) w]_{\alpha} \\
 &\leq \| \psi(a-a(p,0)) w\|_{\infty} + [\psi]_{\alpha} \| a - a(p,0) \|_{\infty} \| w \|_{\infty} \\
 &+ \| \psi\|_{\infty} [ a - a(p,0)]_{\alpha} \| w \|_{\infty} + \|\psi\|_{\infty} \|a - a(p,0)\|_{\infty} [w]_{\alpha},
\end{split}
\end{equation}
where the norms involved are taken over the support of $\psi$.  Write $p = (0, y_0, z_0)$, and for any $q = (x,y,z)$ in the support of $\psi$ 
we find that by interpolating frozen coefficients and using the mean-value theorem,
as well as the fact that $a \in \hoone$
\begin{align*}
|a(x,y,z,t) - a(0,y_0,z_0,0)| 
&\leq \|\partial_x a\|_{\infty} |x| + \|\partial_y a\|_{\infty} |y-y_0| + \|a\|_\A |t|^{\alpha/2},\\
&\leq 2 \|a\|_{1+\A}( d_M(p,q) + |t|^{\alpha/2} ) \leq 2 \|a\|_{1+\A} ( \vep + |T|^{\alpha/2} ).
\end{align*}
where we have used the fact that $a(0,y_0,\cdot,t)$ is constant in $z$.  
So this calculation tells us that there exists some constant $C>0$ such that on the support of $\psi$ 
\[ \|a - a(p,0)\|_{\infty} \leq C( \vep + |T|^{\alpha/2} ). \]

We will also need another observation before we can estimate 
inequality \eqref{eqn:est-key-term-first}.  Observe that 
$w = \Delta H_p [\varphi f]$ vanishes identically at $t=0$. 
Indeed, by similar arguments as in the second statement of Theorem \ref{mapping}, 
a regularity improvement can be translated into a time weight, so that in fact
$H_p:\ho \to t^{\frac{\A}{2}}\mathcal{C}^2_{\textup{ie}}$, where $\mathcal{C}^2_{\textup{ie}}$
is defined exactly as $\hho$ with $\ho$ replaced by bounded continuous functions. Consequently,
we may estimate $\| x^2\Delta H_p [\varphi f] \|_\infty \leq C t^{\frac{\A}{2}} \| f\|_\infty$
Consequently $x^2 w = x^2\Delta H_p [\varphi f]$ vanishes identically at $t=0$, 
and hence same holds for $w$. Since $w \in \ho$, we find for fixed $q$,
\[ 
|w(q,t)| = |w(q,t)-w(q,0)| \leq [w]_{\alpha} |t|^{\alpha/2}. 
\]
We can now estimate the four terms on the right hand side
of \eqref{eqn:est-key-term-first}. 
For the first term in \eqref{eqn:est-key-term-first}, we find using \eqref{Hf}
\begin{equation}\label{E-01}
\| \psi(a-a(p,0)) w\|_{\infty} \leq C( \vep + T^{\alpha/2} ) T^{\alpha/2} \| \varphi f\|_{\alpha}.
\end{equation}
For the second term in \eqref{eqn:est-key-term-first}, we find
\begin{align}\label{E-02}
 [\psi]_{\alpha} \| a - a(p,0) \|_{\infty} \| w \|_{\infty} \leq C \vep^{-\alpha} ( \vep + T^{\alpha/2} ) T^{\alpha/2} \| \varphi f\|_{\alpha},
\end{align}
since the scaling in $[\psi]_{\alpha} \approx \vep^{-\alpha}$.
For the third term appearing in \eqref{eqn:est-key-term-first}, we estimate
\begin{align} \label{E-03}
\| \psi\|_{\infty} [ a - a(p,0)]_{\alpha} \| w \|_{\infty} \leq C [a]_{\alpha} T^{\alpha/2} \| \varphi f\|_{\alpha},
\end{align}
and finally, for the final term in \eqref{eqn:est-key-term-first}, we estimate
\begin{align} \label{E-04}
 \|\psi\|_{\infty} \|a - a(p,0)\|_{\infty} [w]_{\alpha} 
\leq C ( \vep + T^{\alpha/2} ) \| \varphi f\|_{\alpha},
\end{align}
In our eventual application we take $T$ and $\vep$ small, so we may assume $ T < 1$ and $\vep < 1$.  In view of \eqref{E-01}--\eqref{E-04}, as well as \eqref{eqn:est-key-term-first}, we find that indeed if $T/\vep^2 < 1$, then
$\|E^0 f\|_{\alpha} \leq C (\vep + T^{\alpha/2}) \| \varphi f \|_{\alpha}$, 
concluding the proof of the theorem.
\end{proof}

Adjusting the parameters $\vep$ and $T$, we can now establish the mapping
properties of a boundary parametrix $Q_b$, which we have introduced
at the beginning of the subsection.

\begin{prop} \label{prop:boundary-parametrix} For every $\delta > 0$, 
there exists  $\vep > 0$ sufficiently small and $T_0 > 0$ such that
\[ Q_b: \ho(M \times [0,T_0]) \to \hho(M \times [0,T_0]) \]
is a bounded operator, and 
\[ (\partial_t - a \Delta) (Q_b f) = \Phi f + E^0 f + E^1 f, \]
where $\| E^0 f\| < \delta$ and $\| E^1 f \| \to 0$ as $T_0 \to 0$.
\end{prop}
\begin{proof}
Using the previous lemma we may write
\begin{align*}
( \partial_t - a \Delta ) (Q_b f) &= (\partial_t - a \Delta) \left( \sum_{j=1}^N 
\sum_{p \in L_{j,\vep}} \psi_{j,p} H_p [ \varphi_{j,p} f ] \right) \\
&= \sum_{j=1}^N \sum_{p \in L_{j,\vep}}  (\partial_t - a \Delta) \left[\psi_{j,p} H_p [ \varphi_{j,p} f ]   \right] \\
&= \Phi f + \sum_{j=1}^N \sum_{p \in L_{j,\vep}} E^0_{j,p} f + \sum_{j=1}^N \sum_{p \in L_{j,\vep}} E^1_{j,p} f,
\end{align*}
and for $i = 0, 1$  we define $E^i: \ho(M \times [0,T_0]) \to \ho(M \times [0,T_0])$ by
\begin{align*}
E^i f := \sum_{j=1}^N \sum_{p \in L_{j,\vep}} E^i_{j,p} f. 
\end{align*}
It remains to study their operator norms, which by definition are defined by
\begin{align*}
 \|E^i\| &= \sup_{\|f\|_\A =1} \| E^i f \|_{\A} 
 = \sup_{\|f\|_\A =1} \left\| \sum_{j=1}^N \sum_{p \in L_{j,\vep}} E^i_{j,p} f \right\|_{\A}.
 \end{align*}
Given our localization of the H\"older norms, and the respective local definition 
of the $E^i_{j,p} f$, bounding the operator norm of $E^i$ equivalent to estimating the quantity
\[ \sup_{\|f\|_\A =1} \| E^{i}_{j,p} f\|_{\alpha}, \]
independently of $j$ and $p$. The estimate for $E^1$ follows from the previous lemma.
Regarding the estimate for $E^0$, Lemma \ref{lemma:main-frozen-est} gives that 
$\|E^0_{j,p} f\|_{\alpha} \leq C (\vep + T^{\alpha/2}) \| \varphi_{j,p} f \|_{\alpha}$. 
Recalling the scaling that defines $\varphi_{j,p}$ we find $[\varphi_{j,p}]_{\alpha} \leq C \vep^{-\alpha}$ and so
\[ \| E^0\| \leq C (\vep^{1-\alpha} + \vep^{-\alpha} T^{\alpha/2}). \]
Recall that $\alpha \in (0,1)$.  Given $\delta > 0$, choose $\vep > 0$ so 
small that $C \vep^{1-\alpha} < \frac{1}{2} \delta$ and that the level set 
$x = \vep$ is a smooth hypersurface.  Then choose $T > 0$ sufficiently 
small so that both $\vep^{-\alpha} T^{\alpha /2} < \frac{1}{2} \delta$ and $T/\vep^2 < 1$. This proves
the desired estimate for the operator norm of $E^0$.
\end{proof}

\subsection{Parametrix construction} \label{sec:reg-parametrix}

Equipped with the boundary parametrix construction from the previous 
subsection, we construct the full parametrix to the inhomogeneous Cauchy problem.
We begin by constructing an approximate interior parametrix.  This relies on the classical 
theory of parabolic PDE on compact manifolds, and so we only briefly sketch the idea.
\medskip

Recall that $\Phi$ is a bump function that is identically $1$ in a (now fixed) 
$\vep$-neighbourhood of the boundary.  Recall also that $x = \vep$ is a smooth 
hypersurface, and so $Y_{\vep} = \{ x \geq \vep / 2\}$ is a smooth manifold with boundary.  
Let $\overline{Y}$ denote the double of $Y_{\vep}$, which is now a manifold without boundary. 
The Riemannian metric on $Y_{\vep} \subset M$ is not product near the boundary, so that it does
not double to a smooth metric on $\overline{Y}$. We smoothen out the metric in a narrow collar
neighborhood of the join, such that the metric on $\overline{Y}$ and $M$ coincide over $Y_{2\vep}$. 
Note that norms of the edge H\"older spaces are equivalent to the 
classical H\"older spaces since we are working away form the edge.  \medskip

The function $(1-\Phi)$ defines a smooth cutoff function on the closed double $\overline{Y}$,
which we denote by $(1-\Phi)$ again. Now consider the extension $\overline{P}$ of $P=\partial_t + a \Delta$ to 
a uniformly parabolic operator on $\overline{Y}$ and consider the inhomogeneous Cauchy problem
\begin{align*}
\overline{P} u = (1-\Phi) f, \quad u(t = 0) = 0.
\end{align*}
Classical parabolic PDE theory implies the existence of solution operator 
$\widetilde{Q_i}: \mathcal{C}^{\A}(\overline{Y} \times [0,T]) \to 
\mathcal{C}^{2+\A}(\overline{Y} \times [0,T])$. 
Finally, let $\Psi$ be any smooth cutoff function where 
$\Psi \equiv 1$ on $\supp( 1-\Phi )$.  Then define the interior parametrix to be
\[ Q_i f = \Psi \widetilde{Q_i} [ (1-\Phi) f ]. \]
We are now ready to construct the parametrix, which comprises the 
boundary and the interior parametrices introduced above.  Given $f \in \ho(M \times [0,T])$, set
\[ Q f = Q_b f + Q_i f. \]

\begin{prop} \label{prop:short-time-reg}
Let $a$ be a positive function bounded from below that lies in 
$\hoone(M \times [0,T])$, and consider the operator $P = \partial_t + a \Delta$.  
For $T_0>0$ sufficiently small there exists a right inverse $\mathcal{Q}$ for $P$ such that
\[ \mathcal{Q}: \ho(M \times [0,T_0]) \to \hho(M \times [0,T_0]) \]
is a bounded map and if $f \in \ho(M \times [0,T])$ then $u = \mathcal{Q}f$ is a solution to the equation
\[ (\partial_t + a \Delta) u = f, \quad u(p,0) = 0. \]
\end{prop}
\begin{proof}
Given $f \in \ho(M \times [0,T])$, we apply 
Proposition \ref{prop:boundary-parametrix} to compute
\[ (\partial_t + a \Delta) Q f = \Phi f + E^0 f + E^1 f + (1-\Phi f) + E^2 f, \]
where $E^2 f := [\Psi, a \Delta] \left((1-\Phi) f\right)$.  
As in the calculation of $\| E^1\|$, it follows from Lemma 
\ref{lemma:aux-smoothing} that $\| E^2\| \to 0$ as $T \to 0$.  
This in turn shows that the error term $E := E^0 f + E^1 f + E^2 f$ 
can be made to have operator norm strictly less than $1$ for $T_0$ 
sufficiently small.  Thus we may invert $I + E$, acting on $\ho(M \times [0,T_0])$ 
via a Neumann series, and the required right inverse is
\[ \mathcal{Q} = Q ( I + E)^{-1}. \]
\end{proof}

A similar parametrix construction may be used to construct a right 
inverse to the homogeneous Cauchy problem as well.  However 
the following proposition will be sufficient for our purposes, 
despite being non optimal.

\begin{prop} \label{prop:short-time-reg-HCP}
Let $a$ be a positive function bounded from below that lies in 
$\hoone(M \times [0,T])$, and consider the operator $P = \partial_t + a \Delta$.  
For $T_0 > 0$ sufficiently small there exists a right inverse $\mathcal{R}$ for $P$ such that
\[ \mathcal{R}: \hhospace(M) \to \hho(M \times [0,T_0]) \]
is a bounded map and if $u_0 \in \hhospace(M)$ then $u = \mathcal{R} u_0$ 
is a solution to the equation
\[ (\partial_t + a \Delta) u = 0, \quad u(p,0) = u_0. \]
\end{prop}

\begin{proof}
Note that $a \Delta u_0 \in \ho$ for $u_0 \in \hhospace(M)$. Using the right inverse for the inhomogeneous problem, we set 
$\mathcal{R} u_0 := u_0 - \mathcal{Q}(a \Delta u_0)$, which solves the problem, since $u_0$ is time-independent.
Note that the value of $T_0$ is the same as in Proposition \ref{prop:short-time-reg}.
\end{proof}

Finally, we prove Proposition \ref{prop:mainreg}, which extends the existence results
from the shorter time interval $[0,T_0]$ to the full time interval $[0,T]$.

\begin{proof}[Proof of Proposition \ref{prop:mainreg}]
By Proposition \ref{prop:short-time-reg} there exists a $T_0 > 0$ and a 
solution $u \in \hho( M \times [0,T_0])$ to the parabolic initial value problem
\[ (\partial_t + a \Delta) u = f; \; \; \; u(p,0) = 0. \]
If $T_0 \geq T$, then the proof is complete.
Otherwise $T_0 < T$, and we consider the homogeneous Cauchy problem
\begin{equation*}
 (\partial_t + a \Delta) v_1 = 0, v_1(p,0) = u(p,T_0), 
\end{equation*}
where the initial data $u(p,T_0) \in \hhospace(M)$.  
By Proposition \ref{prop:short-time-reg-HCP} the solution to this problem 
exists on the time interval $[0,T_0]$ independent of the initial value $u(p,T_0)$.  
We may also solve
\begin{equation*}
 (\partial_t + a \Delta) u_1 = f(p, t+T_0), u_1(0) = 0,
\end{equation*}
on the interval $[T_0, 2T_0]$, and then the function
\[ \widetilde{u}(p,t) = \begin{cases}
  u(p,t) & \text{for $0 \leq t \leq T_0$} \\
  u_1(p, t-T_0) + v_1(p,t-T_0) & \text{for $T_0 < t \leq 2T_0$}\\
\end{cases}
\]
extends $u$ past $T_0$.  This process continues until $nT_0 > T$, 
and produces a solution $u$ in $\hho(M \times [0,T])$.
Using a modification of the maximum principle proved earlier we 
may deduce uniqueness of solutions which completes the proof 
of Proposition \ref{prop:mainreg}.
\end{proof}

We conclude the subsection with an observation extending the 
statement of Proposition \ref{prop:mainreg}. Observe that the 
mapping properties of the parametrix $Q$, obtained by freezing coefficients, follows directly from the mapping properties of the heat operator
$e^{t\Delta}$ in Theorem \ref{mapping}. Consequently, we actually have the 
following statement, which while not optimal is sufficient for our 
purposes.

\begin{cor} \label{prop:mainreg2}
Let $a(p,t)$ be a positive function bounded from below away from zero
and of H\"older regularity $\hho$. Consider the operator $P = \partial_t + a \Delta$.  
Then the right parametrix $Q$ for $P$ 
constructed in Proposition \ref{prop:mainreg} admits the following mapping property
$$
Q:  \mathcal{C}^{2+\A}_{\textup{ie}} (M\times [0,T]) \to 
\mathcal{C}^{4+\A}_{\textup{ie}} (M\times [0,T])
$$
is bounded.
\end{cor}

\begin{proof}
The only critical point in the construction is the higher order estimate of the norm of $E^0$ in Proposition \ref{prop:boundary-parametrix} which occurs through Lemma \ref{lemma:main-frozen-est}.

Continuing in the notation of Lemma \ref{lemma:main-frozen-est} and dropping the subscripts on $\varphi, \psi$ and the error terms $E^i$, we find we must estimate $ \| \Delta H_p [\varphi f] \|_{2+\alpha}$ in terms of $\| \varphi f \|_{2+\alpha}$.  The tedious estimation is similar to Lemma \ref{lemma:main-frozen-est} and we only remark that it is essential to use the time decay properties 
\[ \|\Delta H_p [\varphi f]\|_{1+\alpha} \leq C T^{\alpha/2}, \|\Delta H_p [\varphi f]\|_{2} \leq C T^{\alpha/2}, [\Delta H_p [\varphi f]]_{2+\alpha} \leq C, \]
and the observation that the H\"older semi-norm of the highest derivatives of $\varphi f$ are only paired with the $L^{\infty}$ norm of $|a-a(p,0)|$, which decays on the support of $\psi$.  
\end{proof}

\subsection{An application of Proposition \ref{prop:mainreg}} \label{sec:appl-reg}
Finally we conclude this section with a first application of our regularity result that will be used below.  

\begin{prop} \label{du-dt} \label{prop:dudt}
Assume $\scal(\ginit)\in \mathcal{C}^{4+\A}_{\textup{ie}}(M)$. Then a positive and bounded from below away from
zero solution $u\in \hho(M\times [0,T])$ to the normalized Yamabe flow 
\begin{align} 
\partial_t u - \frac{(m-1)}{N} u^{1-N} \Delta u = \frac{c(m)}{N} \left( \rho \, u - \scal(\ginit) u^{2-N}\right) 
\end{align}
is in fact $\mathcal{C}^{4+\A}_{\textup{ie}}(M\times [0,T])$ whenever it exists.
\end{prop}

\begin{proof}
Treat the right hand side of this equation as a fixed element of $\hho(M \times [0,T])$, where we note
that $\rho\in \hho([0,T])$ by the transformation formulae in \eqref{transformation} and \eqref{scal-rho-t}.
Since  $\frac{(m-1)}{N} u^{1-N} \in \mathcal{C}^{2+\A}_{\textup{ie}}$ is positive and uniformly 
bounded away from zero by assumption, we may apply the parabolic regularity result in Corollary 
\ref{prop:mainreg2} to obtain a solution $v \in \mathcal{C}^{4+\A}$ satisfying the equation
 \[ \partial_t v - \frac{(m-1)}{N} u^{1-N} \Delta v = \frac{c(m)}{N} \left( \rho \, u - \scal(\ginit) u^{2-N}\right). \]
Note that $w:= u-v$ solves $\partial_t w - \frac{(m-1)}{N} u^{1-N} \Delta w = 0$ with zero initial condition.
By the maximum principle $\partial_t w_{\max} \leq 0$ and $\partial_t w_{\min} \geq 0$. Due to the 
initial condition $w(0)=0$, we deduce $w\equiv 0$ and hence $u=v \in \mathcal{C}^{4+\A}(M \times [0,T])$. 
\end{proof}

\section{Uniqueness of the Yamabe flow on singular edge spaces}\label{unique-section}

We now provide an argument to show that solutions to the Yamabe flow 
\eqref{YF-conf} are unique. Uniqueness of solutions to the normalized Yamabe flow 
is then an easy consequence by a rescaling in time. We will use the parabolic maximum principle stated in 
Corollary \ref{PMP2}. 

\begin{thm}
Consider the Yamabe flow \eqref{YF-conf} for the conformal factor $u$  given by
\[ u^{\frac{4}{m-2}} \partial_t u = (m-1) \Delta u - \frac{m-2}{4} \scal(\ginit) u, \]
where $\Delta$ denotes the Laplace Beltrami operator for $\ginit$. One we specify initial data, a solution $u$ in $\hho$ is unique.
\end{thm}

\begin{proof} Suppose $u$ and $v$ are two positive solutions to this equation in 
$\hho$ that satisfy the same initial condition.  Set $\w = u-v$.  Then $\w(p,0) = 0$ and
\begin{align*}
u^{\frac{4}{m-2}} \partial_t u - v^{\frac{4}{m-2}} \partial_t v &= 
(m-1) \Delta \w - \frac{m-2}{4} \scal(\ginit) \w. 
\end{align*}
We find for the evolution of $\w$
\begin{align*}
\partial_t \w &= u^{-\frac{4}{m-2}}  \left((m-1) \Delta \w  - \frac{m-2}{4} \scal(\ginit) \w + 
(u^{\frac{4}{m-2}} - v^{\frac{4}{m-2}}) \partial_t v\right) \\
&= \left( (m-1) u^{-\frac{4}{m-2}} \right) 
\Delta \w - \frac{m-2}{4} \w\,  u^{-\frac{4}{m-2}} \scal(\ginit)  \\
&+ \w\, u^{-\frac{4}{m-2}} \int_0^1 
\left( su + (1-s) v \right)^{\frac{6-m}{m-2}} ds,
\end{align*}
where we have used Taylor's theorem in the last equality.  
Abstractly we have shown that $\w$ satisfies a parabolic equation
$\partial_t \w = a \Delta \w + b \w,$
where $a, b \in \hho$ are H\"older functions, and $a > 0$. 
We now apply an integrating factor trick.  For some constant $c$ to be chosen 
let $z = e^{c t} \w$.  Then $z(0) = w(0) = 0$ and $z$ satisfies the following parabolic equation
\[ \partial_t z = a \Delta z + (b+c)z. \]
We would like the coefficient $b+c$ to be negative, so choose 
$c < -\sup |b|$, then $b+c < 0$.  From the maximum principle we conclude 
$z \leq 0$, and from this we find that $w \leq 0$.  Repeating the 
argument by switching the roles of $u$ and $v$ yields $w = 0$.  
So solutions to the Yamabe flow are unique in $\hho$.
\end{proof}

\section{Evolution of the scalar curvature under the Yamabe flow}\label{scal-section}

In this subsection we review some general geometric preliminaries for the Yamabe flow.  
Everything here is well known in the classical case on a compact manifold without 
boundary, see for example \cite{Chow} and \cite{Ye}.  For background on the Yamabe problem in general, see \cite{LeeParker}. \medskip

From now on when discussing edge Yamabe flow, we will always assume feasibility 
of the edge metric $\ginit$ and H\"older regularity of the scalar curvature 
$\scal (\ginit)\in \hhh$, such that $\scal (\ginit)$ and $\Delta \scal (\ginit)$ are $\hho$. By Proposition 
\ref{du-dt} this implies in particular that $\Delta u, \partial_t u\in \hho$.

\subsection{Total scalar curvature functional and the conformal Yamabe invariant}
Recall the initial feasible edge metric is denoted by $\ginit$, its Laplacian by $\Delta$ and the normalized Yamabe flow evolves the metric 
within its conformal class with $g = u^{\frac{4}{m-2}} \ginit$. 
Consider the total scalar curvature functional 
\begin{align*}
\cS(g) &:= \frac{1}{\vl(g)^{\frac{m-2}{m}}} \int_M \scal(g) \; \dv.
\end{align*}
The volume form, the volume
and the scalar curvature of $g=u^{\frac{4}{m-2}} \ginit$ are given by the following expressions ($\Delta$ denotes 
the Laplace Beltrami operator defined with respect to $\ginit$)
\begin{equation}\label{transformation}
\begin{split}
&\dv = u^{\frac{2m}{m-2}} \textup{dvol}_{\ginit},  \\
&(\vl(g))^{\frac{m-2}{m}} = \left(\int_M u^{\frac{2m}{m-2}} \textup{dvol}_{\ginit} \right)^{\frac{m-2}{m}} 
= \left\| u \right\|^{2}_{L^{\frac{2m}{m-2}}}, \\
&\scal(g) = u^{-\frac{m+2}{m-2}}\left( -4 \frac{m-1}{m-2} \Delta u +  
\scal(\ginit) u \right),
\end{split}
\end{equation}
where we set $c(m):= \frac{m+2}{4}$ and employed the notation for $L^{\frac{2m}{m-2}}(M,\ginit)$ norms in the second equation. 
From there we compute for the total scalar curvature functional 
\begin{align*}
\cS(g) &= \frac{1}{\vl(g)^{\frac{m-2}{m}}} \int_M \scal(g) \; \dv \\
&= \frac{1}{\left\| u \right\|^{2}_{\frac{2m}{m-2}}} \int_M u \left( -4 \frac{m-1}{m-2} 
\Delta u +  \scal(\ginit) u \right) \textup{dvol}_{\ginit} \\
&= \frac{1}{\left\| u \right\|^{2}_{\frac{2m}{m-2}}} \int_M -4 \frac{m-1}{m-2} u 
\Delta u +  \scal(\ginit) u^2  \textup{dvol}_{\ginit} \\
&= \frac{1}{\left\| u \right\|^{2}_{\frac{2m}{m-2}}} \int_M 4 \frac{m-1}{m-2} 
| \nabla u |^2 +  \scal(\ginit) u^2  \textup{dvol}_{\ginit}.
\end{align*}
The final equation is obtained after integration by parts, and there are no boundary 
terms, since $u\in \hho$ lies in the domain of the self-adjoint Friedrichs extension of the Laplacian.
A crucial fact is that for $u\in \hho$ and $\ginit$ feasible the expressions above are all bounded
despite a singularity of the Riemannian metric. Moreover, due to H\"older's inequality
\[ \left| \int_M \scal(\ginit) u^2  \dvinit \right| \leq \| \scal(\ginit) 
\|_{L^{\frac{m}{2}}} \cdot \| u \|_{L^{\frac{2m}{m-2}}}^2. \]
We conclude that, independently of the sign of scalar curvature, 
for $g$ conformal to $\ginit$ via $u \in \hho$, the total scalar curvature is bounded from below
\[ \cS(g) > -\infty. \]
Consequently, we find that the conformal invariant
\begin{align}\label{Yamabe-invariant-definition} 
\nu([\ginit]) = \inf \left\{ \cS(g): g = u^{\frac{4}{m-2}} \ginit, u \in 
\hho \right\}
\end{align}
is bounded from below.  Note that the `true' Yamabe invariant, where the infimum is 
taken over any larger set of positive functions appears to be unbounded because of the 
generally singular nature of curvature. 

\subsection{Evolution of the scalar curvature}
Before we proceed with geometric considerations concerning the scalar curvature, 
consider the Yamabe flow equation \eqref{YF-conf} for the conformal factor
$g(t) = u^{\frac{4}{m-2}}(t) \ginit$ more closely. Our analysis in \cite{BV} implies that $u\in \hho$. 
By Proposition \ref{du-dt} in fact we have even $\Delta u, \partial_t u \in \hho$. \medskip

Straightforward computations now give the following closed expressions 
for the evolution of $\scal(g(t))$ and the average scalar curvature $\rho(t)$, introduced in the 
normalization \eqref{NYF} (not to be confused with the total scalar curvature functional $\cS(g)$) 
along the normalized Yamabe flow
\begin{equation}\label{scal-rho-t}
\begin{split}
&\partial_t \scal(g) = (m-1) \Delta \scal(g) + \scal(g) (\scal(g) - \rho), \\
&\partial_t \rho = -\frac{m-2}{2 \vol(g)} \int_M (\scal(g) - \rho)^2 \dv.
\end{split}
\end{equation}
In particular, the average scalar curvature $\rho(t)$
decreases along the flow. 

The next result is classical.
Note that regularity of $\scal(g) \in \hho$ follows from
the transformation formula \eqref{transformation} and the fact $\Delta u \in \hho$.
 
\begin{lemma}\label{decrease}
If the initial scalar curvature $\scal (\ginit)\in \hhh$ is negative, bounded away from zero, 
i.e. $\scal(\ginit) < -b < 0$ for some fixed constant $b>0$, then 
the maximum of $\scal(g(t)) < 0$ decreases along the normalized Yamabe flow.
\end{lemma}
\begin{proof}
Consider the function $\scal_{\max}(t) = \max_M \scal(g(t))$, which is
continuous and satisfies for $t>0$ the differential inequality (we use $\scal(g) \in \hho$
and the maximum principle in Theorem \ref{maxprin})
\[ \partial_t \scal_{\max} \leq \scal_{\max} (\scal_{\max} - \rho). \]
The quantity $\w (t) := \scal_{\max}(t) - \rho(t)$ is always non-negative.  We therefore have 
the differential inequality
\[ \partial_t \scal_{\max} \leq \w \, \scal_{\max}, \]
which integrates for any $\vep \in (0,t)$ to
\begin{align*}
\scal_{\max} (t) \leq \exp \left(\int^t_{\vep} \w(s) ds \right) \scal_{\max}(\vep). 
\end{align*}
Taking $\vep \to 0$, this implies the maximum scalar curvature is always negative.  Plugging this 
fact into the differential inequality
\[ \partial_t \scal_{\max} \leq \scal_{\max} (\scal_{\max} - \rho) \leq 0, \]
we conclude that in fact $\scal_{\max}$ decreases along the flow.
\end{proof}
Next, we show that in fact $\scal(g)$ approaches $\rho$ along the flow at an exponential 
rate. This result is well known in the surface case. Our proof uses strongly the maximum 
principle obtained in Theorem \ref{maxprin}.

\begin{prop} \label{prop:scalar-curv} If $\scal (\ginit)\in \hhh$
and there exists positive constants $a$ and $b$ where
\[ - a < \scal( \ginit ) < -b < 0, \]
then for a solution to the normalized Yamabe flow in $\hho(M \times [0,T])$,
\[ \| \scal(g(t)) - \rho(t) \|_{\infty, M} \leq C e^{-bt} \]
with the constant $C>0$ being independent of $T$.
\end{prop}

\begin{proof}
Consider the function $\scal_{\min}(t) = \min_M \scal(g(t))$
and the inequalities for the maximum and minimum scalar 
curvatures
\begin{align*}
&\partial_t \scal_{\max} \leq \scal_{\max} (\scal_{\max} - \rho), \\
&\partial_t \scal_{\min} \geq \scal_{\min} (\scal_{\min} - \rho). 
\end{align*}
Upon subtracting we obtain
\begin{align*}
\partial_t (\scal_{\max} - \scal_{\min} ) 
&\leq \scal_{\max} (\scal_{\max} - \rho) - \scal_{\min} (\scal_{\min} - \rho)\\
&\leq -b (\scal_{\max} - \rho) - \rho (\scal_{\min} - \rho) \\
&= -b (\scal_{\max} - \scal_{\min} ) + (\rho + b) (\rho - \scal_{\min}) \\
&\leq -b (\scal_{\max} - \scal_{\min} ),
\end{align*}
where in the second inequality we estimated the first summand 
on the right hand side using the fact that $(\scal_{\max} - \rho) \geq 0$,
$\scal$ is decreasing along the flow (by Lemma \ref{decrease}) and hence is bounded from above by $(-b)$.
For the second summand we used $(\scal_{\min} - \rho) \leq 0$. The last inequality
follows from $(\scal_{\min} - \rho) \leq 0$ and $(\rho + b) \leq 0$, with the latter being
a consequence of $\rho$ decreasing along the flow. \medskip

Integrating the inequality we find as in the proof of the previous Lemma
\begin{align}\label{maxmin}
(\scal_{\max} - \scal_{\min} ) \leq c_0 e^{-bt},
\end{align}
where $c_0$ depends only on the initial data.
Consequently it suffices to 
prove an appropriate lower bound for $(\scal_{\min}-\rho)$. 
The equations above allow us to write an evolution equation for $\scal - \rho$:
\[ \partial_t ( \scal - \rho ) = (m-1) \Delta \scal + \scal( \scal - \rho) +\frac{m-2}{2} 
\int_M (\scal(g) - \rho)^2 \dv. \]
Using the maximum principle and the fact that $(\scal(g) - \rho)^2\geq 0$, we may estimate
\[ \partial_t (\scal_{\min} - \rho) \geq \scal_{\min} (\scal_{\min} - \rho) \geq -b 
(\scal_{\min} - \rho), \]
where in the second inequality we used the fact that $(\scal_{\min} - \rho) \leq 0$, 
the scalar curvature decreases along the flow and
the initial scalar curvature is bounded from above by $(-b)$. Integrating the inequality 
proves an exponential lower bound for $(\scal_{\min} - \rho)$ and the statement now follows 
in view of \eqref{maxmin}. 
\end{proof}

\section{Uniform estimates of solutions to the edge Yamabe flow}\label{uniform-section}

In this section we establish a priori estimates for the solution $u\in \hho(M \times [0,T))$ of equation \eqref{flow2},
where we assume that $T< \infty$ is finite. If $T=\infty$, normalized Yamabe flow exists
for all times $t>0$ and uniform estimates are obsolete. \medskip
 
\subsection{Uniform estimate of $u$}
In order to obtain $L^{\infty}$ estimates, Ye \cite{Ye} used the 
elliptic maximum principle to obtain a differential inequality for 
the maximum and minimum values of the conformal factor as a function 
of time, which was then explicitly integrated to obtain the desired 
bounds.  We follow this approach and adapt Ye's proof to the edge setting.

\begin{prop} \label{prop:L-infinity-estimate} Suppose that $u$ is a 
maximal solution to the normalized Yamabe flow in $\hho(M \times [0,T))$.  
Suppose $\scal(\ginit) \in \hho(M)$ is negative bounded away from zero. 
Then there exists a constant $c > 0$, depending on $u(0), \rho(0), \max |\scal(\ginit)|$ 
and $\min |\scal(\ginit)|$, and being independent of the maximal existence time $T$, such that
$c^{-1}\leq u(p,t) \leq c$ for all $p \in M$ and $t\in [0,T)$.
\end{prop}

\begin{proof}
For what follows we use the fact that $u_{\min}(t)$, $u_{\max}(t)$ are positive.
This is clear from the fact that $u^{\frac{2}{m-2}} = e^{v}$, where
$v \in \hho$ is the solution to the second equation in \eqref{YF-conf}
obtained in \cite{BV}. Let us assume that $u(t)$ attains its maximum at 
$p_{\max}(t) \in \widetilde{M}$, and its minimum at $p_{\min}(t) \in \widetilde{M}$.
Note 
\begin{align*}
&\partial_t u_{\max}(t) \equiv \partial_t u(p_{\max}(t), t) = 
\partial_1 u(p_{\max}(t), t) + (\partial_t u)(p_{\max}(t), t) 
\leq (\partial_t u)(p_{\max}(t), t), \\
&\partial_t u_{\min}(t) \equiv \partial_t u(p_{\min}(t), t) = 
\partial_1 u(p_{\min}(t), t) + (\partial_t u)(p_{\min}(t), t) 
\geq (\partial_t u)(p_{\min}(t), t).
\end{align*}
Hence the maximum principle in Theorem \ref{maxprin},
applied to \eqref{flow2} yields
\begin{equation}\label{est}
\begin{split}
&\frac{d u_{\min}^N (t) }{dt} \geq c(m) \min |\scal(\ginit)| \, u_{\min}(t) + c(m) \rho \, u_{\min}^N(t), \\
&\frac{d u_{\max}^N (t)}{dt} \leq c(m) \max |\scal(\ginit)| \, u_{\max}(t) + c(m) \rho \, u_{\max}^N(t).
\end{split}
\end{equation}
\textbf{Estimates for the minimum function $u_{\min}$.}
The key point is this is almost a linear differential inequality in $u_{\min}^{N-1}$. 
To see this we rewrite the left hand side in the first inequality in \eqref{est} as 
\begin{equation}
N u_{\min}^{N-1} \frac{d u_{\min} }{dt} \geq c(m) \min |\scal(\ginit)| u_{\min}(t) + c(m) \rho \, u_{\min}^N(t), 
\end{equation}
and then divide by $u_{\min}$ (note that $u_{\min}$ is positive so that this does not change 
the sign of the inequality) to see that
\begin{equation}
N u_{\min}^{N-2} \frac{d u_{\min} }{dt} \equiv \frac{N}{N-1} \frac{d u_{\min}^{N-1} }{dt}  
\geq c(m) \min |\scal(\ginit)| + c(m) \rho \, u_{\min}^{N-1}(t).
\end{equation}
This is now a linear differential inequality in $w = w(t) = u_{\min}^{N-1}(t)$

\begin{equation}
\begin{split}
w' &\geq \min |\scal(\ginit)| + \rho \, w \\
&\geq \min |\scal(\ginit)| + \rho(0) \, w=:  a + b w,
\end{split}
\end{equation}
where in the last inequality we used the fact that $\rho(t)$ decreases along the flow.
An inequality of the form $w' - bw \geq a$ can be rewritten as 
$(e^{-bt} w)' \geq a e^{-bt}$ and integrated over $[0,t]$. This gives 
\[ e^{-bt} w(t) - w(0) \geq \frac{a}{b} (1 - e^{-bt} ). \]
Writing out $a,b$ and $w$ we find (note $\rho(0)<0$)
\begin{equation}\label{estimate1}
\begin{split}
u_{\min}^{N-1}(t) &\geq u_{\min}^{N-1}(0) e^{\rho(0) t} + \frac{\min |\scal(\ginit)|}{\rho(0)} 
(e^{\rho(0) t}-1) \\ &\geq u_{\min}^{N-1}(0) + \frac{\min |\scal(\ginit)|}{|\rho(0)|}(1 - e^{\rho(0) t}).
\end{split}
\end{equation}
\textbf{Estimates for the maximum function $u_{\max}$.}
Consider the second inequality in \eqref{est}. Using the fact that the total 
scalar curvature $\rho(t)$ is negative and decreasing along the flow, while $u_{\max}$
is positive, we may estimate
\begin{equation}
\frac{d u_{\max}^N (t)}{dt} \leq c(m) \max |\scal(\ginit)| \, u_{\max}(t) + c(m) \rho \, u_{\max}^{N}.
\end{equation}
Dividing both sides of the inequality as before by $u_{\max}$, and using again the fact 
that $\rho(t)$ is decreasing along the flow, we obtain
\begin{equation}
 \frac{d u_{\max}^{N-1} }{dt}  \leq \frac{N}{N-1} \frac{d u_{\max}^{N-1} }{dt}  
\leq c(m) \max |\scal(\ginit)| + c(m) \rho(0) \, u_{\max}^{N-1}.
\end{equation}
Let us write $v = u_{\max}^{N-1}$ and 
abbreviate the  equation above as $v' \leq A + Bv$. As before, 
this inequality can be integrated to give 
\[ v(t) \leq \left(v(0) + \frac{A}{B}\right) e^{Bt} - \frac{A}{B}. \]
Note that $B = c(m) \rho(0) < 0$, and $A/B =\max |\scal(\ginit)|/\rho(0) < 0$. 
Consequently, $e^{Bt}<1$ and we find by reinserting $v = u_{\max}^{N-1}$ into the equation
\begin{align}\label{estimate2}
u_{\max}^{N-1}(t) \leq u_{\max}^{N-1}(0) + \frac{\max |\scal(\ginit)|}{|\rho(0)|}. 
\end{align}
\end{proof}

Note that our estimates rule out growth of solutions with time. 

\subsection{Uniform estimates of the time derivative of $u$}

Our next result establishes uniform bounds for the time derivative of $u$.
The proof works exactly the same way in the setting of compact smooth 
manifolds and does not employ the Krylov-Safonov estimates, which were a key  ingredient in the corresponding argument of Ye \cite{Ye} in the derivation of 
higher order uniform bounds. Surprisingly, such an observation has not been 
made elsewhere in the literature.

\begin{prop} \label{bound-time-derivative} Suppose that $u$ is a 
maximal solution to the normalized Yamabe flow in $\hho(M \times [0,T))$.  
Suppose $\scal (\ginit)\in \hhh$ is negative bounded away from zero. 
Then there exists a constant $c > 0$, depending on 
$u(0), \rho(0), \scal(\ginit)$ and independent of the maximal existence time $T$, such that
$\| \partial_t u \|_{\infty, M} \leq ce^{-bt}$.
\end{prop}

\begin{proof}
The central point is \emph{not} to try to present the normalized Yamabe flow
equation as a parabolic equation. We consider \eqref{NYF} and 
write it as an equation for the conformal factor with $g(t)=u^{\frac{4}{m-2}} \ginit$.
Dividing both sides of the flow equation by\footnote{Note that $u$ is shown 
to be bounded away from zero.} $L \, u^{L-1}$ where we write $L=\frac{4}{m-2}$,
we obtain
\begin{align}
\partial_t u = \frac{m-2}{4} \bigl(\rho(t) - \scal(g(t))\bigr) u.
\end{align}
Taking supremum norms in space on both sides, we find
\begin{align}
\| \partial_t u \|_{\infty, M} = \frac{m-2}{4} 
\| \rho(t) - \scal(g(t)) \|_{\infty, M}  \cdot \| u \|_{\infty, M}.
\end{align}
In view of the estimates established in Proposition \ref{prop:scalar-curv} 
and Proposition \ref {prop:L-infinity-estimate}, we conclude the statement.
\end{proof}

\subsection{Uniform H\"older regularity of $u$}
We can now establish uniform H\"older regularity of $u$.
Linearizing \eqref{flow2} for $u$ around $u(0)=1$, we write $u=1+u'$ and find
\begin{equation}\label{u-hoelder-eqn}
\begin{split}
(\partial_t - (m-1)\Delta) u' = Q\{\scal (\ginit), \rho(t), u', \partial_t u'\},
\end{split}
\end{equation}
where $Q$ is some polynomial combination of the terms in the brackets
which have uniform time-independent $L^\infty$ bounds by Propositions 
\ref{prop:L-infinity-estimate} and \ref{bound-time-derivative}. Rescaling time, we 
conclude from Proposition \ref{mapping2}
\begin{align}
u=e^{t\Delta}Q\{\scal (\ginit), u', \partial_t u'\} \in \mathcal{C}^{1+\A}_{\textup{ie}}(M\times [0,T]).
\end{align} 
We can in fact conclude more. Add $(m-1) u'$ on both sides of the equation 
\eqref{u-hoelder-eqn} and obtain
\begin{equation}
\begin{split}
\left(\partial_t - (m-1)(\Delta-1)\right) u' = Q'\{\scal (\ginit), \rho(t), u', \partial_t u'\},
\end{split}
\end{equation}
where as before, the term $Q'$ on the right hand side of the linearized equation
is still $L^\infty$ bounded with a time-independent bound. Writing 
$u=e^{t(\Delta-1)}Q'\{\scal (\ginit), u', \partial_t u'\}$, we may in 
fact conclude that the H\"older norm of $u$ is bounded uniformly in $T$.
Indeed, $(-\Delta+1)$ is discrete with spectrum $\geq 1$ and we can conclude that
pointwise the heat kernel is converging exponentially to zero as $t\to \infty$.
As a consequence of this, one can repeat the H\"older space estimates from 
Proposition \ref{mapping2} to conclude that $e^{t(\Delta-1)}: \mathscr{B}(M\times [0,T]) \to 
\mathcal{C}^{1+\A}_{\textup{ie}}(M\times [0,T])$ is a bounded operator
with operator norm bounded independently of $t$ and $T$. We have proved the following theorem.

\begin{thm}\label{u-hoelder}
Assume $\scal (\ginit)\in \hhh(M)$ is negative bounded away from zero.
The solution $u\in \hho(M\times [0,T))$ in fact is 
$u\in \mathcal{C}^{1+\A}_{\textup{ie}}(M\times [0,T])$ with a bound on its H\"older norm
being independent of the maximal existence time $T$.
\end{thm}

\section{Long-time existence of the edge Yamabe flow}\label{long-section}

We are ready to complete the long-time existence argument. First, we employ 
a parabolic regularity argument in Proposition \ref{prop:mainreg} to get a regularity jump 
for the solution $u$ from $\mathcal{C}^{1+\A}_{\textup{ie}}$ to $\hho$. 
This proceeds parallel to the argument of Proposition \ref{prop:dudt}. 
Rewrite flow equation \eqref{flow2} as
\begin{align}\label{flow-3}
\partial_t u - \frac{(m-1)}{N} u^{1-N} \Delta u = \frac{c(m)}{N} \left( \rho \, u - \scal(\ginit) u^{2-N}\right). 
\end{align}
Treat the right hand side of this equation as a fixed element of $\ho(M \times [0,T])$, where we note that $\rho\in \ho([0,T])$.
Since  $\frac{(m-1)}{N} u^{1-N} \in \mathcal{C}^{1+\A}_{\textup{ie}}$ is positive and uniformly 
bounded away from zero, we may apply the parabolic regularity result in Proposition \ref{prop:mainreg} to 
obtain a solution $v \in \hho$ with initial condition $v(0)=1$
 \[ \partial_t v - \frac{(m-1)}{N} u^{1-N} \Delta v = \frac{c(m)}{N} \left( \rho \, u - \scal(\ginit) u^{2-N}\right). \]
Note that $w:= u-v$ solves $\partial_t w - \frac{(m-1)}{N} u^{1-N} \Delta w = 0$ with zero initial condition.
By the maximum principle $\partial_t w_{\max} \leq 0$ and $\partial_t w_{\min} \geq 0$. Due to the 
initial condition $w(0)=0$, we deduce $w\equiv 0$ and hence $u=v \in \hho(M \times [0,T])$. 
Now apply Proposition \ref{prop:dudt} to conclude that $u \in \hhh(M \times [0,T])$. \medskip

Consider $u_0= u(T) \in \hhh(M)$. We wish to restart the Yamabe flow \eqref{flow2}
with the initial condition $u(0)=u_0$. Consider $e^{t\Delta}$ acting on $\hhh(M)$ without convolution
in time. Then $e^{t\Delta}$ maps $\hhh(M)$ to $\hhh(M\times \R^+)$ without any gain in H\"older regularity,
since due to absence of time integration, the estimates in \cite{BV} proceed without two additional front face powers 
in the asymptotics of the integral kernels on $\mathscr{M}^2_h$. Note that $\Delta u_0 \in \hho(M)$ and
by uniqueness of solutions to the heat equation $\Delta e^{t\Delta} u_0 = e^{t\Delta} \Delta u_0 \in \hho(M \times \R^+)$. Consequently, 
$\partial_t e^{t\Delta} u_0 = e^{t\Delta} \Delta u_0 \in \hho(M \times \R^+)$. \medskip

We write $u=u' + e^{t\Delta} u_0$ and plug this into the Yamabe flow equation \eqref{flow2},
with rescaled time $\tau= (t-T)$. Linearizing around $u'$, we find as before in
Proposition \ref{prop:dudt}
\begin{align}
\left[ \partial_t - (m-1) (e^{t\Delta}u_0)^{1-N}\Delta\right] u'=
Q_1(u') + Q_2(u',\partial_t u'), \quad u'(0)=0,
\end{align}
where $Q_1$ and $Q_2$ denotes linear and quadratic combinations of the 
elements in brackets, respectively, with coefficients given by polynomials
in $e^{t\Delta}u_0$ and $\partial_t e^{t\Delta}u_0, \Delta e^{t\Delta} u_0$.
Since these coefficients are of higher H\"older regularity $\hho(M)$, we may set
up a contraction mapping argument in $\hho$ and extend $u$ past the maximal 
existence time $T$ ad verbatim to the proof of Proposition \ref{prop:dudt}. 
This proves long-time existence. 

\begin{thm}\label{long-thm} 
Suppose $\ginit$ is a feasible edge metric, such that its scalar curvature $\scal(\ginit)\in \hhh(M)$ is negative, and hence satisfies $- a < \scal( \ginit ) < -b < 0$ for some  
positive constants $a$ and $b$. Then the normalized Yamabe flow \eqref{NYF} admits a solution
$g(t) = u^{\frac{4}{m-2}} \ginit$ with $u\in \hho (M \times [0,\infty))$.
\end{thm}

\section{Convergence of the edge Yamabe flow}\label{convergence-section}

We now prove the convergence part of Theorem \ref{BV}.

\begin{thm}\label{convergence-thm} 
Suppose $\ginit$ is a feasible edge metric such $\scal(\ginit)\in \hhh(M)$ and  $\scal(\ginit) < 0$. Suppose that $u(t)$ is a solution to the normalized Yamabe flow \eqref{NYF} that exists for all time.  
Then the associated metric $g(t) = u^{\frac{4}{m-2}} \ginit$ converges to a metric with constant negative curvature.
\end{thm}
\begin{proof}
Consider the associated metric $g(t) = u^{\frac{4}{m-2}} (t) \ginit$.  This solves the normalized Yamabe flow, 
and as $g$ exists for all time we have by Proposition \ref{prop:scalar-curv} that
\[ \scal(g(t)) - \rho(t) \rightarrow 0 \]
at an exponential rate.  Consequently, by the Yamabe flow equation
\[ \partial_t g(t) \rightarrow 0 \]
exponentially. We conclude that $g$ converges to a 
continuous limit metric $g^* = (u^*)^{\frac{4}{m-2}} \ginit$, and the conformal 
factor $u(t)$ admits a continuous pointwise limit $u^*$ as $t\to \infty$. 
We prove the statement by proving that $u^* \in \mathcal{C}^{2+\A}_{\textup{ie}}(M)$.
Then the limit metric $g^*$ admits a well-defined scalar curvature, which can be 
shown to be constant. \medskip

Theorem \ref{u-hoelder} yields uniform bound of $u(t)$ and hence also of
$u(t)^{-1}$ in $\mathcal{C}^{1+\A}_{\textup{ie}}(M)$
for $t\in \R^+$. To get even further uniform regularity improvement, 
we consider the Yamabe flow equation \eqref{flow-3} 
rewritten as
\[ \partial_t u - \frac{(m-1)}{N} u^{1-N} (\Delta-1) u = \frac{c(m)}{N} 
\left( \rho \, u - \scal(\ginit) u^{2-N}\right) + \frac{(m-1)}{N} u^{2-N}. \]
Now, in a similar way as in the proof of Theorem \ref{u-hoelder}, 
we conclude that the parametrix $Q$ of the parabolic operator 
on the left hand side with uniform $\mathcal{C}^{1+\A}_{\textup{ie}}(M \times \R^+)$
coefficient is pointwise converging exponentially to zero as $t\to \infty$, 
since the heat operator of $(\Delta-1)$ behaves the same way as well.
Consequently, $u(t) \in \hho$ uniformly in $t\in \R^+$. \medskip

Choose now any divergent sequence $t_n \to \infty$.  
Uniform $\mathcal{C}^{2+\A}_{\textup{ie}}(M)$ bounds on the sequence $\{u(t_n)\}$
imply existence of a convergent subsequence in the Banach space $\mathcal{C}^{2+\A'}_{\textup{ie}}(M)$, for some $\A' < \A$.
Consequently the pointwise limit $u^*$ lies in $\mathcal{C}^{2+\A'}_{\textup{ie}}(M)$ and the 
limit metric $g^*$ admits a well-defined scalar curvature.  By 
$\scal(t)-\rho(t)$ vanishing in the limit $t\to \infty$, we conclude that the scalar curvature of the limit metric 
must be constant and negative by Lemma \ref{decrease}.
\end{proof}

\section{Metrics with negative Yamabe invariant}\label{negative-section}

In this section we extend a classical result of conformal geometry to our singular setting.  Consider a feasible incomplete edge metric $g$.  Recall we denote the conformal Laplacian of $g$ by 
\[ \Box^g u := - \frac{4(m-1)}{(m-2)} \, \Delta^g u+  \scal (g) u.\]  The associated Yamabe invariant, introduced in 
\eqref{Yamabe-invariant-definition} is given by
\begin{align}\label{YI}
\nu([g]) &= \inf \left\{ \cS(\widetilde{g}) \mid \widetilde{g} = u^{\frac{4}{m-2}} g, u \in 
\mathcal{C}^{2+\A}_{\textup{ie}}(M) \right\} = \inf_{u \in \hho(M)} \frac{\langle \Box^g u, u\rangle_{L^2}}{\| u \|^2_{L^{\frac{2m}{m-2}}}}.
\end{align}

We say that a conformal class $[g]$ is scalar positive, negative or zero if $\nu([g])$ is positive, negative or zero, respectively.  We restate and prove Theorem \ref{negative}.

\begin{thm}
Let $g$ be a feasible incomplete edge metric with conformal Laplacian $\Box^g$ and $\scal(g) \in \hhh(M)$.  The following are equivalent.
\begin{enumerate}
\item The first eigenvalue of $\Box^g$ is scalar positive (respectively negative or zero).
\item There exists a metric $\widetilde{g} = u^{\frac{4}{m-2}} g$ such that $\scal(\widetilde{g}) > 0$\\
 (respectively $<0$ or $=0$).
\item $[g]$ is positive (respectively negative or zero).
\end{enumerate}
\end{thm}

\begin{proof} We begin with a description of the first eigenvalue of 
$\Box^g$ in this setting.  We then prove the equivalence of the three statements. 
\medskip

We fix a domain for $\Box^g$ by $\dom(\Box^g) := \dom(\Delta)$, where $\Delta$ denotes the 
Friedrichs self-adjoint extension of the negative Laplace Beltrami operator in $L^2(M,g)$.
Since by assumption $\scal (g)$ is bounded, $\dom(\Box^g)$ defines a self-adjoint extension 
for the conformal Laplacian $\Box^g$. We still write $\Box^g$ for its self-adjoint extension. Note that
$(-\Delta)$ is positive and hence $\Box^g$ is bounded from below. Let $\{ \lambda_j \}_{j\in \N}$ be 
an ascending enumeration of eigenvalues of $\Box^g$. By a classical result for self-adjoint operators in a Hilbert space that are bounded from below with discrete spectrum, the first eigenvalue $\lambda_1$ 
admits a characterization as a Rayleigh quotient (cf. \cite[Theorem XIII.1]{Reed})
\begin{equation}\label{Rayleigh}
\begin{split}
\lambda_1 = \inf_{u \in \dom(\Box^g)} \frac{\langle \Box^g u, u\rangle_{L^2}}{\| u \|^2_{L^2}}.
\end{split}
\end{equation}
The infimum is attained by the corresponding eigenfunction $\phi_1$.
We now prove that $\phi_1 \in \mathcal{C}^{2k+\A}_{\textup{ie}}(M)$ for any $k\in \N$.
Under the assumption $\scal(g)\in \hhh$ the heat kernel construction for $\Box^g$ follows along the lines of the heat kernel
construction for $\Delta$ in \cite{MazVer:ATO}. In particular, the microlocal heat kernel description in Theorem 
\ref{heat-expansion} holds for the conformal Laplacian $\Box^g$ as well. We begin with the following observation
\begin{align*}
\left|\phi_1(p)\right|^2 &= e^{-2\lambda_1} |e^{-\Box}\phi_1(p)|^2 \\ &=
e^{-2\lambda_1} \left| \int_M e^{-\Box}(p, q) \phi_1(q) \textup{dvol}_{g}(q)\right|^2
\\ &\leq e^{-2\lambda_1} \int_M (e^{-\Box}(p, q))^2 \textup{dvol}_{g}(q) 
\int_M (\phi_1(q))^2 \textup{dvol}_{g}(q)
\end{align*}
Note that the heat kernel $e^{-t\Box}$ appears below with $t=1$ and hence a straightforward 
estimate shows that $\phi_1$ is bounded up to the edge singularity. Propositions \ref{mapping}
and \ref{mapping2} now yield that $\phi_1 = e^{-\lambda_1}e^{-\Box}\phi_1 \in 
\mathcal{C}^{2+\A}_{\textup{ie}}(M)$. Using the relation $\phi_1 = e^{-\lambda_1}e^{-\Box}\phi_1$
iteratively, proves the regularity statement $\phi_1 \in \mathcal{C}^{2k+\A}_{\textup{ie}}(M)$ for any $k\in \N$.
\medskip

We now prove that $\phi_1$ can be assumed to be nowhere vanishing in the open interior of $M$ and positive.
We proceed in several steps. We first show that $\phi_1$ cannot change sign and hence can be assumed
to be non-negative on $M$. We then prove that $\phi_1$ must be strictly positive as an application of the 
maximum principle in Theorem \ref{maxprin}. \medskip

Let us assume that $\phi_1$ changes sign in $M$ and hence its absolute value $|\phi_1|$ is discontinuous.
On the other hand, $|\phi_1|$ still minimizes the right hand side of \eqref{Rayleigh} and it is straightforward that it is again an 
eigenfunction of $\Box$. Since eigenfunctions of $\Box$ are smooth in the open interior $M$ by elliptic
regularity, we conclude that $\phi_1$ does not change sign in the open interior of $M$ and hence 
we may assume without loss of generality that $\phi_1$ is non-negative on $M$. Excluding zeros of
$\phi_1$ is in fact more intricate and we refer to \cite[Proposition 1.15]{ACM}, where this is proved.
\medskip

With these preliminaries aside, we can begin the main part of the proof which we believe is classical \cite{GR-notes}.
\medskip

\textbf{(1. $\Longrightarrow$ 2.)}  Let $\lambda_1$ denote the first eigenvalue of $\Box^g$ as above with corresponding eigenfunction $\phi_1$.  Since $\phi_1 > 0$ we consider the scalar curvature of the conformally related metric $\widetilde{g} = \phi_1^{\frac{4}{m-2}} g$:
\eqref{transformation}
\begin{align*}
\scal(\widetilde{g}) = \phi_1^{-\frac{m+2}{m-2}}\left( - \frac{4(m-1)}{(m-2)} \Delta \phi_1 +  
\scal(g) \phi_1 \right) = \lambda_1 \phi_1^{-\frac{4}{m-2}}.
\end{align*}
We conclude that the sign of the scalar curvature of $\widetilde{g}$ matches the sign of $\lambda_1$.
Note that $\widetilde{g}$ is a feasible edge metric of of regularity $\mathcal{C}^{4+\A}_{\textup{ie}}(M)$.
The fact that $\widetilde{g}$ is again feasible follows by the arguments of \S \ref{invariance} after a change of
the (edge) boundary defining function $x$. Here we point out that $\phi_1$ is polyhomogeneous 
in the sense of Definition \ref{phg}, since $\phi_1 = e^{-\lambda_1}e^{-\Box}\phi_1$ and the heat kernel 
of $e^{-t\Box}$ is polyhomogeneous as well. Consequently, $\phi_1$ admits an asymptotic expansion 
at the edge in the strong sense with smooth coefficients.  Since $\phi_1 \in \mathcal{C}^{4+\A}_{\textup{ie}}(M)$, strictly bounded away from zero, the scalar curvature is of same regularity. \medskip

\textbf{(2. $\Longrightarrow$ 3.)} Consider the metric $\widetilde{g}$ of fixed sign.  We consider the cases separately.
If $\scal(\widetilde{g}) < 0$, then consider the test function $u \equiv 1$.  Clearly $\nu([g]) = \nu([\widetilde{g}]) < 0$.  In case $\scal(\widetilde{g}) = 0$, then the same test function shows that $\nu([g]) = \nu([\widetilde{g}]) \leq 0$.  However, in this case if $\nu([g]) < 0$, then for some admissible $u$ we have the contradiction that $\int_M |\nabla u|^2 \textup{dvol}_{g} < 0$, from which we conclude $\nu([g]) = 0$. Finally, if $\scal(\widetilde{g}) > 0$, then there exist constants $C, C' > 0$ where for any admissible $u$
\[ \langle \Box^g u, u\rangle_{L^2} \geq C \| u \|_{H^{1,2}(M)} \geq C' \|u\|_{L^{\frac{2m}{m-2}}}, \]
where we have used the Sobolev embedding $L^{\frac{2m}{m-2}}(M) \subset H^{1,2}(M)$.  Consequently $\nu([g]) \geq C' > 0$.
\medskip

\textbf{(3. $\Longrightarrow$ 1.)}
If $\nu([g]) > 0$, then let $\widetilde{g}$ be the conformal multiple corresponding to eigenfunction $\phi_1$ of the first eigenvalue $\lambda_1$ of $\Box^g$, normalized so that $\|\phi_1\|^2_{L^{\frac{2m}{m-2}}} = 1$.  But then by the characterization of the first eigenvalue, 
\[ \lambda_1 \| \phi_1 \|_{L^2}^2 = \inn{ \Box \phi_1}{\phi_1}_{L^2} \geq \nu([g]) > 0, \]
proves that $\lambda_1 > 0$. If $\nu([g]) < 0$, then for some admissible function $v > 0$, 
\[ \frac{\inn{ \Box v}{v}_{L^2}}{\| v \|^2_{L^{\frac{2m}{m-2}}}} < 0. \]
Thus $\inn{ \Box v}{v}_{L^2} < 0$, and so $\lambda_1 < 0$ by the Rayleigh characterization.
Finally if $\nu([g]) = 0$, then repeating the analysis of the first case using $\phi_1$ we may conclude that $\lambda_1 \geq 0$.  Arguing as above, $\lambda_1 > 0$ would entail that $\nu([g]) > 0$, and so we conclude $\lambda_1 = 0$.  This completes the proof.

\end{proof}

\def\cprime{$'$}
\providecommand{\bysame}{\leavevmode\hbox to3em{\hrulefill}\thinspace}
\providecommand{\MR}{\relax\ifhmode\unskip\space\fi MR }
\providecommand{\MRhref}[2]{%
  \href{http://www.ams.org/mathscinet-getitem?mr=#1}{#2}
}
\providecommand{\href}[2]{#2}

\begin{thebibliography}{ACM11}


\bibitem[\textsc{AkBo03}]{AkutagawaBotvinnik}
K.~Akutagawa and B.~Botvinnik, \emph{Yamabe metrics on cylindrical manifolds},
  Geom. Funct. Anal. \textbf{13} (2003), no.~2, 259--333. \MR{1982146
  (2004e:53051)}

\bibitem[\textsc{ACM12}]{ACM}
Kazuo Akutagawa, Gilles Carron, and Rafe Mazzeo, \emph{The {Y}amabe problem on
  stratified spaces}, Geom. Funct. Anal. \textbf{24} (2014), no. 4, 1039--1079 \MR{MR3248479}


\bibitem[\textsc{BaVe13}]{BV}
Eric Bahuaud and Boris Vertman, \emph{Yamabe flow on manifolds with edges},
  Math. Nachr. \textbf{287} (2014), no. 2--3, 127--159. \MR{3163750}

\bibitem[\textsc{BDV11}]{BDV}
E.~Bahuaud, E.~Dryden, and B.~Vertman, \emph{Mapping properties of the heat
  operator on edge manifolds}, Math. Nachr.  \textbf{288} (2015), no. 2--3, 126--157. \MR{3310503}



\bibitem[\textsc{Bre11}]{Brendle}
Simon Brendle, \emph{Evolution equations in {R}iemannian geometry}, Jpn. J.
  Math. \textbf{6} (2011), no.~1, 45--61. \MR{2835361}

\bibitem[\textsc{Bre07}]{BrendleYF}
Simon Brendle, \emph{Convergence of the {Y}amabe flow in dimension 6 and higher}, Invent. Math. \textbf{170} (2007), no.~3, 541--576. \MR{2357502}

%
%
%
%
%
\bibitem[\textsc{CLN06}]{Chow}
B. Chow, Peng Lu and Lei Ni, \emph{Hamilton's Ricci flow}, Grad. Studies Math. 
\textbf{77}, AMS. Providence RI (2006) 

\bibitem[\textsc{Don11}]{Donaldson}
S.~K. Donaldson, \emph{{K}\"ahler metrics with cone singularities along a
  divisor}, Essays in Math. and Appl., p.49-79, Springer Heidelberg (2012)

\bibitem[\textsc{EpMa13}]{EM}
C. Epstein and R. Mazzeo, \emph{Degenerate diffusion operators arising in population biology}, Annals of Mathematics Studies, {\bf{185}}  Princeton University Press, Princeton, NJ, 2013. xiv+306 pp. ISBN: 978-0-691-15715-3 

\bibitem[\textsc{Eva10}]{Evans}
L.~C. Evans, \emph{Partial differential equations}, 2nd edition Grad. Studies in Math. \textbf{19} 
AMS. Providence RI (2010)


%


\bibitem[\textsc{GiTo10}]{Topping}
Gregor Giesen and Peter~M. Topping, \emph{Ricci flow of negatively curved
  incomplete surfaces}, Calc. Var. Partial Differential Equations \textbf{38}
  (2010), no.~3-4, 357--367. \MR{2647124 (2011d:53155)}

\bibitem[\textsc{GiTo11}]{Topping2}
\bysame, \emph{Existence of {R}icci flows of incomplete surfaces}, Comm.
  Partial Differential Equations \textbf{36} (2011), no.~10, 1860--1880.
  \MR{2832165}


\bibitem[\textsc{Gri01}]{Gr}
\textsc{D. Grieser}, \emph{Basics of the {$b$}-calculus}, Approaches to singular analysis ({B}erlin, 1999).  
Oper. Theory Adv. Appl. \textbf{125} (2001), 30--84. \MR {1827170 (2002e:58051)}

%

\bibitem[\textsc{Jef05}]{Jeffres}
\textsc{T. Jeffres}, \emph{A maximum principle for parabolic equations on manifolds with cone singularities}, 
Adv. Geom. \textbf{5} (2005), no. 2 319-323


\bibitem[\textsc{JeLo03}]{JefLoy:RSH}
\textsc{T. Jeffres} and \textsc{P. Loya}, \emph{Regularity of solutions of the heat equation on a cone}, 
Int. Math. Res. Not.  2003,  no. 3, 161--178. \MR{1932532 (2003i:58043)}

\bibitem[\textsc{JMR11}]{JMR}
Thalia Jeffres, Rafe Mazzeo, and Yanir Rubinstein, \emph{{K}\"ahler-{E}instein
  metrics with edge singularities}, Ann. of Math. (2) \textbf{183} (2016) no. 1, 95--176. \MR{3432582}

\bibitem[\textsc{JeRo10}]{JeffresRowlett}
Thalia Jeffres and Julie Rowlett, \emph{Conformal deformations of conic metrics
  to constant scalar curvature}, Math. Res. Lett. \textbf{17} (2010), no.~3,
  449--465. \MR{2653681 (2011e:53046)}


\bibitem[\textsc{KrSa80}]{KrylovSafonov}
N. V. Krylov and M. V. Safonov, \emph{A property of the solutions of parabolic equations with 
measurable coefficients}, Izv. Akad. Nauk SSSR Ser. Mat. textbf{44} (1980), no.~1,
  161-175.
	
%
%
\bibitem[\textsc{LePa87}]{LeeParker}
John~M. Lee and Thomas~H. Parker, \emph{The {Y}amabe problem}, Bull. Amer.
  Math. Soc. (N.S.) \textbf{17} (1987), no.~1, 37--91. \MR{888880 (88f:53001)}

\bibitem[\textsc{LSU67}]{LSU:LAQ}
O.~A. Lady{\v{z}}enskaja, V.~A. Solonnikov, and N.~N. Ural{\cprime}ceva,
  \emph{Linear and quasilinear equations of parabolic type}, Translated from
  the Russian by S. Smith. Translations of Mathematical Monographs, Vol. 23,
  American Mathematical Society, Providence, R.I., 1967. \MR{0241822 (39
  \#3159b)}


\bibitem[\textsc{MRS11}]{MRS}
Rafe Mazzeo, Yanir Rubinstein, and Natasha Sesum, 
\emph{Ricci flow on surfaces with conic singularities}, 
Anal. PDE 8 (2015), no. 4, 839--882. \MR{3366005}

\bibitem[\textsc{MaVe12}]{MazVer:ATO}
Rafe Mazzeo and Boris Vertman, \emph{Analytic {T}orsion on {M}anifolds with
  {E}dges}, Adv. Math. \textbf{231} (2012), no. 2, 1000--1040 \MR{2955200}

\bibitem[\textsc{Mel93}]{Mel:TAP}
Richard~B. Melrose, \emph{The {A}tiyah-{P}atodi-{S}inger index theorem},
  Research Notes in Mathematics, vol.~4, A K Peters Ltd., Wellesley, MA, 1993.
  \MR{1348401 (96g:58180)}

%

\bibitem[\textsc{Pol88}]{GR-notes}
Daniel Pollack, \emph{Topics in Differential Geometry}, Lecture notes taken by Daniel Pollack for a 1988 graduate class taught by Richard Schoen. \texttt{www.math.washington.edu/~pollack/research/Schoen-1988-notes.html.}

\bibitem[\textsc{ReSi78}]{Reed} \textsc{M. Reed} and \textsc{B. Simon}, 
\emph{Methods of Modern Mathematical Physics IV. Analysis of Operators}, 
Acad. Press New York (1978)

\bibitem[\textsc{Sim02}]{MilesSimon}
M.~Simon, \emph{Deformation of {L}ipschitz {R}iemannian metrics in the
  direction of their {R}icci curvature}, Differential geometry, {V}alencia,
  2001, World Sci. Publ., River Edge, NJ, 2002, pp.~281--296. \MR{1922058
  (2003f:53121)}

\bibitem[\textsc{Sha15}]{Shao}
Yuanzhen Shao, \emph{The Yamabe flow on incomplete manifolds}
preprint on arXiv:1506.07018 [math.AP] (2015)
%
%


  
\bibitem[\textsc{ScRo15}]{Schrohe} 
 Elmar Schrohe and Nicolas Roidos \emph{Existence and maximal $L^p$-regularity 
 of solutions for the porous medium equation on manifolds with conical singularities},
 preprint on arXiv:1504.05101 [math.AP] (2015)
  


\bibitem[\textsc{SS03}]{SS}
Harmut Schwetlick and Michael Struwe, \emph{Convergence of the 
{Y}amabe flow for ``large'' energies}, J. Reine Angew. Math. 
\textbf{562} (2003) 59--100. \MR{2011332 (2004h:53097)}

%

\bibitem[\textsc{Ver16}]{Ricci-Vertman}
Boris Vertman, \emph{Ricci flow on singular manifolds}, [arXiv:1603.06545] (2016)

%
%
\bibitem[\textsc{Ye94}]{Ye}
Rugang Ye, \emph{Global existence and convergence of Yamabe flow}, J. Diff. Geom.
\textbf{39} (1994), no.~1, 35-50.

\bibitem[\textsc{Yin13}]{Yin}
Hao Yin, \emph{Ricci flow on surfaces with conical singularities, II}, preprint 2013. arXiv 1305.4355.

\end{thebibliography}
\end{document}